\documentclass[12pt, reqno]{amsart}
\usepackage{a4,amsthm}
\usepackage{mathabx}
\theoremstyle{plain}
\newtheorem{theorem}{Theorem}[section]
\newtheorem{lem}[theorem]{Lemma}
\newtheorem{prop}[theorem]{Proposition}

\newtheorem{defi}[theorem]{Definition}

\newtheorem{remark}[theorem]{Remark}

\def\ad{{\rm ad}}
\def\Ad{{\rm Ad}}

\def\a{{\mathfrak{a}}}

\def\s{{\mathfrak{s}}}

\def\m{{\mathfrak{m}}}

\def\p{{\mathfrak{p}}}    
\def\k{{\mathfrak{k}}}

\def\g{{\mathfrak{g}}}
\def\l{{\mathfrak{l}}}

\def\C{{\mathbb{C}}}
\def\R{{\mathbb{R}}}

\def\nsmallskip{\smallskip\noindent}
\def\bbigskip{\bigskip\bigskip}
\def\nbigskip{\bigskip\noindent}

\def\nmedskip{\medskip\noindent}

\def\buildunder#1#2{\mathrel{\mathop{\kern0pt #2}
\limits_{#1}}}

\def\dds{\frac{d}{ds}{\big |_{s=0}}}
\def\ddt{\frac{d}{dt}{\big |_{t=0}}}

\def\pn{\par\noindent}
\def\sn{\smallskip\noindent}
\def\mn{\medskip\noindent}
\def\bn{\bigskip\noindent}

\begin{document}

\title[The adapted hyper-K\"ahler  structure]{The adapted hyper-K\"ahler structure on the crown domain}

\bbigskip

\author[Laura Geatti]{Laura Geatti}
\author[Andrea Iannuzzi]{Andrea Iannuzzi}

\address{Dipartimento di Matematica,
Universit\`a di Roma  ``Tor Vergata", Via della Ricerca Scientifica 1,
I-00133 Roma, Italy} 
\email{geatti@mat.uniroma2.it, iannuzzi@mat.uniroma2.it}



\begin{abstract}
Let  $\,\Xi\,$ be the crown domain associated with a non-compact irreducible hermitian symmetric space
$\,G/K$.
We give an explicit description of  the unique $\,G$-invariant {\it adapted hyper-K\"ahler
structure} on $\,\Xi$,$\ $i.$\,$e.$\ $compatible with the adapted complex structure $\,J_{ad}\,$  and 
with the $\,G$-invariant K\"ahler structure of $\,G/K$.
 We also compute invariant potentials of the 
involved K\"ahler metrics  and the associated moment maps.
\end{abstract} 
\maketitle

\centerline{Roma, 5 novembre 2017}

%-------------------------------------------------------------------
%-----------------INTRO---------------------------------------
%-------------------------------------------------------------------

\bigskip
\section{Introduction}

\bigskip
\noindent

 A quaternionic complex structure on a 4n-dimensional real  manifold consists of 
 three complex structures $\,I,\ J,\ K\,$ such that $\,IJK=-Id$.
 It is called  hyper-K\"ahler if there exist 
 2-forms $\,\omega_I$, $\,\omega_J$, $\,\omega_K$  which are 
K\"ahler for  $\,I,\ J,\ K$, respectively, and define the same Riemannian
metric  given by 
$$g (\,\cdot\,,\,\cdot\, )=\omega_I (\,\cdot\,,I\,\cdot\, )=\omega_J(\,\cdot\,,J\,\cdot\,)=
 \omega_K(\,\cdot\,,K\,\cdot\, )\,.$$ 
 A hyper-K\"ahler manifold is holomorphic symplectic with respect to any
 of its complex structures, e.g. the complex symplectic form 
 $\,\omega_J + i \omega_K\,$ is holomorphic with respect to $\,I$.

Let $\,(G/K,\, g_0)\,$ be an irreducible Hermitian symmetric space.
In  \cite{BiGa96a}, O. Biquard and B. Gauduchon  proved that in the compact case the holomorphic cotangent bundle $\,T^*G/K^{1,0}$, endowed with its canonical holomorphic symplectic form $\,\omega^\C_{can}$, carries a unique $\,G$-invariant hyper-K\"ahler metric whose restriction to $\,G/K$, identified with the zero section,
 coincides with~$\,g_0$ (see also \cite{Cal79}).
They also showed that in the non-compact case  such a hyper-K\"ahler metric only exists on an appropriate tubular neighbourhood of
$\,G/K$ in $\,T^*G/K^{1,0}$. 

Identify $\,T^*G/K^{1,0}\cong T^*G/K\,$ with the tangent bundle 
$\,TG/K\,$ via the metric$\,g_0$. For $\,G/K\,$ a classical Hermitian symmetric space,
A. S. Dancer and R. Sz\"oke (\cite{DaSz97}) have shown that the 
 hyper-K\"ahler metric constructed in \cite{BiGa96a}
is determined by $\,\omega^\C_{can}\,$ and the pull-back
of the so-called adapted complex structure $\,J_{ad}\,$  (see \cite{LeSz91} and \cite{GuSt91})
 via a  $\,G$-equivariant fiber preserving diffeomorphism of the tangent bundle $\,TG/K$.
This suggests that, on the maximal domain of existence of $\,J_{ad}$, 
there exists a $\,G$-invariant hyper-K\"ahler  structure which includes $\,J_{ad}$.

In this paper we consider an arbitrary non-compact Hermitian symmetric space $\,G/K.$
We regard the maximal domain of existence of the adapted complex structure as a 
$\,G\,$-invariant domain $\,\Xi\,$  in the Lie group complexification $\,G^\C/K^\C$ of $\,G/K$,
where $\,J_{ad}\, $ coincides with the complex structure of $\,G^\C/K^\C$ (cf. \cite{BHH03}).
In the literature $\,\Xi\,$ is referred to as the crown domain associated with~$\,G/K$.

We show that indeed $\,\Xi\,$ admits a unique  $\,G$-invariant {\it 
adapted hyper-K\"ahler  structure},
i.e.  such that $\,J=J_{ad}\,$ and the  restriction  of
$\,(I,\, \omega_I)\,$ to $\,G/K\,$ coincides with 
the K\"ahler structure  defined by  $\,g_0$.
The adapted hyper-K\"ahler  structure coincides with the pull-back of the hyper-K\"ahler  structure determined by O. Biquard and P. Gauduchon. However, it satisfies different   initial conditions
and its uniqueness does not follow from their arguments.
Moreover, from the condition $\,J=J_{ad}\,$  it is easy to deduce  that the  forms $\,\omega_I\,$ and
$\,\omega_K\,$ are   locally $\,G^\C$-invariant     (Lemma 7.2), a fact which was not evident from the previous  investigations. 

For all the quantities involved in the adapted hyper-K\"ahler  structure, we provide  explicit formulas  in Lie theoretical terms. 
In the case of $\,G/K\,$  compact, one can adopt
 a similar strategy to obtain a  unique invariant adapted hyper-K\"ahler  structure on the whole complexification $\,G^\C/K^\C \cong TG/K$.
 
In order to state our main result we need to fix some notation.
Let $\,\k \oplus \p\,$ be the Cartan decomposition of the Lie algebra $\,\g\,$ of $\,G\,$
with respect to $\,K$.
Let $\,\a\,$ be a maximal abelian subalgebra  of $\,\p\,$ and denote 
by $\,\Sigma\,$ the associated restricted root system.  
The  crown domain associated with $\,G/K\,$ in $\,G^\C/K^\C\,$ is by definition 
$$\,\Xi =G \exp i \Omega K^\C/K^\C\,,$$
where
$\textstyle \Omega:=\{ H \in \a \ : \ |\alpha(H)|<{\pi\over 2}, ~\forall \alpha\in\Sigma \}\,$
is the  cell defined by D. N. Akhiezer and S. G. Gindikin in  \cite{AkGi90}.
The closed subset $\,\exp i \Omega K^\C/K^\C\,$ is a $\,G$-slice of~$\,\Xi$.

Let $\,I_0\,$ be the $\,G$-invariant  complex structure  of $\,G/K$. On $\,\p \cong T_{eK} G/K$, it coincides  with the adjoint action of a central 
 element of $\,\k\,$ (see (\ref{CENTER})). 
 Its $\,\C$-linear extension to $\,\p^\C\,$ is also denoted by $\,I_0$.
 The conjugation with respect to $\,\p$ of an element $\,Z\,$ in $\,\p^\C\,$
is indicated  by $\,\overline Z$.

The Killing form of $\,\g^\C$, as well as its
restrictions to $\,\p^\C\,$ and to $\,\p$, is denoted by $\,B$. 
 The  standard $\,G$-invariant K\"ahler structure  $\,(I_0, \omega_0)\,$  on
 $\,G/K\,$ is uniquely determined by its
restriction to $\,\p\,$, namely 
 $\, \omega_0(\, \cdot \,,\, \cdot \,)=B(I_0 \,\cdot\,, \, \cdot\,)$.
Finally, 
for $\,z \in G^\C/K^\C\,$ and  $\,Z\in \g^\C\,$, let 
$$\textstyle
\widetilde Z_z:=\dds \exp(sZ) \cdot z $$
be  the vector field induced by the holomorphic $\,G^\C$-action on $\,G^\C/K^\C$.
Our main  result is as follows.

\bigskip
\noindent
{\bf Theorem.}
{\sl   Let $\,G/K\,$ be an irreducible, non-compact Hermitian symmetric space
endowed with its standard $\,G$-invariant K\"ahler structure, and let $\,\Xi\,$ be the associated crown domain.
There exists a unique  $\,G$-invariant adapted hyper-K\"ahler structure $\,(I,J,K, \omega_I, \omega_J, \omega_K)\,$
on
$\,\Xi$, i.e.  such that $\,J=J_{ad}\,$ and the 
K\"ahler structure $\,(I,\, \omega_I)\,$ coincides with $\,(I_0,\, \omega_0)\,$
when restricted to $\,G/K$. 

\mn
 $\bf (a)$ The symplectic $\,J$-holomorphic form
$\,\omega_I-i \omega_K\,$ is the restriction of a $\,G^\C$-invariant form on $G^\C/K^\C$ 
and  is uniquely determined by $$\,(\omega_I-i \omega_K)(Z,W)=B(I_0Z, \, W)\,,$$
for  $\,Z,\,W \in \p^\C \cong T_{eK^\C}G^\C/K^\C$.

\mn
 $\bf (b)$
For $\,z=aK^\C\,$ on the $\,G$-slice $\,\exp i \Omega K^\C/K^\C\,$ of $\,\Xi$, the $\,G$-invariant complex structure $\,I\,$ is
  given by
$$\,I \widetilde Z_z= \widetilde {\overline {I_0Z}}_z\,.$$

Let $\,\{A_1, \cdots, A_r\}\,$ be  the  basis 
 of $\,\a\,$ defined in  (\ref{NORMALIZ1}) and (\ref{NORMALIZ2}) of Section \ref{PRELIMINARIES}
 and $\,C:=B(A_1,\,A_1)= \dots = B(A_r,\,A_r)$.
Write $\,a = \exp iH$, where  $\,H=  \sum_{j=1}^r t_jA_j$.

 \mn
 $\bf (c)$
Let $\,\rho_0 \,$ be  a potential of $\,\omega_0\,$ and  let $\,p:\Xi \to G/K\,$
 the $\,G$-equivariant projection given by $\,p(gaK^\C)=gK$.
A potential of $\,\omega_I\,$ is given by $\,\rho_0 \circ p +\rho_I$,
where the $\,G$-invariant  function $\,\rho_I\,$ is defined by
$$\,\textstyle \rho_I(gaK^\C):=  -\frac{C}{4}\sum_j f_I(2t_j)\,,$$
with $\,f_I\,$  a real valued function satisfying $\,\frac{\sin x}{\cos x} f_I'(x)= \cos x -1$.

 \mn
 \noindent
 $\bf (d)$
A $\,G$-invariant potential of $\,\omega_J\,$ is given by
$$\, \textstyle \rho_J(ga K^\C):=-\frac{C}{4}\sum_{j=1}^r \cos (2t_j)\,.$$
The  moment map $\,\mu_J: \Xi \to \g^*\,$  associated with $\,\rho_J\,$  is given by
 $$\,\textstyle \mu_J(gaK^\C)(X)=B(\Ad_{g^{-1}}X, \Psi(H))\,,$$ where $\,\Psi(\sum_{j=1}^r t_jA_j)=
\frac{1}{2}\sum_{j=1}^r \sin (2t_j)A_j$.}

\bigskip

%Note that $\,C\,$ is a constant depending only on the symmetric space
%$\,G/K$. 

We sketch the strategy of the proof. If such a $\,G$-invariant hyper-K\"ahler  structure exists, 
then the forms  $\,\omega_I\,$ and $\,\omega_K\,$ are
necessarily restrictions of $\,G^\C$-invariant forms on $\,G^\C/K^\C$
(Lemma \ref{OLOSIMPLECTIC}).
It follows that they coincide with the forms given in (a) (Rem.$\,$\ref{INVARIANTIK}).
A standard argument also shows that they are closed (Lemma \ref{FORMS2}(iii)).

The forms  $\,\omega_I\,$, $\,\omega_K$, the complex structure $\,J=J_{ad}$,
and the  almost complex structure $\,I\,$
defined in (b),
determine  a $\,G$-invariant quaternionic  almost complex structure. 
Then, by a result of N. J. Hitchin,
the integrability of $\,I\,$ and $\,K:=IJ\,$  follows from the closeness
of  $\,\omega_J(\,\cdot\,,\,\cdot\,):=\omega_I(J\,\cdot\,,I\,\cdot\,)\,$  
(\cite{Hit87}, Lemma 6.8). This property is proved 
by showing
 that  the $\,G$-invariant function $\,\rho_J\,$ defined in (d) is a potential  of  $\,\omega_J\,$
by means of restricted root theory and moment map techniques (Prop.~\ref{POTENTIAL}). 
As a result, 
$\,(\,I,\ J,\ K\,,\,\omega_I,\,\omega_J,\,\omega_K)\,$ is a $\,G$-invariant adapted hyper-K\"ahler  structure, as claimed.
A  similar strategy is used to obtain the 
$\,G$-invariant potential $\,\rho_I\,$ of $\,\omega_I-p^*\omega_0\,$ indicated in (c)
(Prop. \ref{POTENTIALII}). Such potential
is expressed in terms of a real function $\,f_I\,$ satisfying 
a  simple trigonometric differential equation
(cf. \cite{BiGa96a}, Thm.$\,$1). A proof of uniqueness of the
adapted hyper-K\"ahler  structure
 is outlined in Section \ref{PROOF}.
In the case of $\,G=SL_2(\R)$, all details  are given in Appendix A. 

As a further application of the above techniques, we also provide
a Lie theoretical formulation of the pull-back to $\,\Xi\,$ 
of the canonical real symplectic form on the cotangent
bundle $\,T^*G/K$ (Appendix B).

 \medskip
The exposition is organized as follows. 
In Section 2 we collect the basic facts which are needed in the sequel.
In Section 3 we introduce the (almost) complex structure $\,I$. 
In Section 4 we express in a Lie theoretical fashion the inverse of the 
$G$-equivariant diffeomorphism introduced in \cite{DaSz97}.
This leads to a useful expression of $\,I\,$ which is exploited
in the computation of $\,2i\partial \bar \partial_J\rho_J$.
In Section 5 we introduce the forms $\,\omega_I,\,\omega_J$, $\,\omega_K\,$
and study their basic properties.
In Section 6 we  show that the   $\,G$-invariant function $\,\rho_J\,$
is a potential of  $\,\omega_J\,$ and compute the associated moment map.
In Section 7 we prove the main theorem by assembling the results obtained in 
other sections. 
In Section 8 we  show that the   $\,G$-invariant function $\,\rho_I\,$
is a potential of  $\,\omega_I-  p^* \omega_0$.

%-------------------------------------------------------------------
%-----------------PRELIMINARIES----------------------------
%-------------------------------------------------------------------

\medskip
\section{Preliminaries}
\label{PRELIMINARIES}

\medskip
Let $\,\g\,$ be a non-compact semisimple Lie algebra and let $\,\k\,$ be a maximal compact subalgebra of $\,\g\,$.
%Let $B( á , á )$ be the  Killing form on $\g$  as well as its complex linear extension to $\g^\C$.
Denote by  $\,\theta\,$  the Cartan involution of $\,\g\,$ with respect to $\,\k$, with Cartan decomposition $\,\g=\k\oplus\p$.
Let  $\,\a\,$ be a maximal abelian subspace in $\,\p$. The dimension of $\,\a\,$ is by definition the {\it rank} of $\,G/K$.
The adjoint action of $\,\a\,$   decomposes $\,\g\,$  as
 $$\g= \a\oplus \m\oplus\bigoplus_{\alpha\in\Sigma}\g^\alpha,$$
 where $\,\m\,$ is the centralizer of $\,\a\,$ in $\,\k$,
 the joint eigenspace 
 $\,\g^\alpha=\{X\in\g~|~ [H,X]=\alpha(H)X, {\rm \ for\  every \ } H\in\a\}\,$ is the $\,\alpha$-restricted root space and  $\,\Sigma\,$ consists of those  $\,\alpha\in\a^*\,$ for which
 $\,\g^\alpha\not=\{0\} $. Denote by $B$ the Killing form of $\g$, as well as its holomorphic extension to $\g^\C$ (which coincides with the Killing form of $\g^\C$).
  
For $\,\alpha\in\Sigma$, consider the $\,\theta$-stable  space $\,\g[\alpha]:=\g^\alpha\oplus  \g^{-\alpha}$, and denote by $\,\k[\alpha]\,$ and
$\,\p[\alpha]\,$ the projections of $\,\g[\alpha]\,$ onto $\,\k\,$ and $\,\p$, respectively. %Since $\theta\g^{\alpha}=\g^{-\alpha}$, the 
 Then  
\begin{equation}\label{DECO}\k=\m \oplus  \bigoplus_{\alpha \in \Sigma^+} \k[\alpha]\, \qquad 
{\rm and} \qquad \p= \,\a  \oplus \bigoplus_{\alpha \in \Sigma^+} \p[\alpha] \, \end{equation}
are $\,B$-orthogonal decompositions of $\,\k\,$ and $\,\p$, respectively.
 %where $\,\p[\alpha] := (\g^{\alpha} \oplus \theta\g^{\alpha}) \cap \p$.
\bigskip
\begin{lem}
\label{BASIS}
%Let $\,G/K\,$ be an irreducible  Hermitian symmetric space.
Every element $\,X\,$ in $\,\p\,$ decomposes in a unique way  as 
 $$X_\a + \textstyle \sum_{\alpha \in \Sigma^+} P^\alpha,$$
where $\,X_\a\in \a\,$ and    $\,P^\alpha\in \p[\alpha]$. 
The vector $\,P^\alpha\,$ can be written uniquely as  $\,P^\alpha=X^\alpha-\theta X^\alpha$, where
$\,X^\alpha\,$ is the component of $\,X\,$ in the root space $\,\g^\alpha$.
Moreover,  $[H,P^\alpha]=\alpha(H)K^\alpha$, where $\,K^\alpha\,$ is the element in $\,\k[\alpha]\,$ uniquely defined by $\,K^\alpha=X^\alpha+\theta X^\alpha$.
\end{lem}

\begin{proof} The proof of this lemma is an easy exercise.\end{proof}

The restricted root system of a Lie algebra $\,\g\,$  of Hermitian type is either of type 
$\,C_r\,$ (if $\,G/K\,$ is of tube type) or of type $\,BC_r\,$ (if $\,G/K\,$ is not of tube type), i.e.  there exists a basis $\,\{e_1,\ldots,e_r\}\,$ of $\,\a^*\,$ for which   $\,\Sigma =\Sigma^+
\cup-\Sigma^+$, with
$$\Sigma^+=\{2e_j, ~1\le j\le r,~~e_k\pm e_l,~ 1\le k< l\le r\},\quad \hbox{ for type
$\,C_r$}, $$
$$\Sigma^+=\{e_j,~2e_j,~1\le j\le r,~~e_k\pm e_l,~~1\le k<l\le r\},\quad \hbox{ for type
$\,BC_r$}\,.$$
With the above choice of a positive system $\,\Sigma^+$, the roots
$$\lambda_1:=2e_1,\,\dots\,,\,\lambda_r:=2e_r \,$$
form a maximal set of  long strongly orthogonal positive restricted  roots (i.e. such that  $\,\lambda_k\pm \lambda_l\not\in \Sigma $, for $\,k\not=l$).

For every $\,j=1,\ldots,r$, the root space   $\, \g^{\lambda_j}\,$ is one-dimensional.
Fix $\,\ E^j \in \g^{\lambda_j}\,$ such that the $\,\s \l (2)$-triples 
$\,\{E^j,~\theta  E^j,~ A_j:=[\theta  E^j,\, E^j]\}\,$
are normalized as follows  
\begin{equation}\label{NORMALIZ1}
[ A_j,\, E^j]=2  E^j, \quad \hbox{for}\quad j=1,\ldots,r. \end{equation}
The vectors $\,\{ A_1,\ldots, A_r\} \,$ form  a $\,B$-orthogonal basis of $\,\a\,$ and 
\begin{equation}\label{NORMALIZ2}
[ E^k,\, E^l]=[ E^k,\,\theta  E^l]=0,\quad [ A_k,\, E^l]=\lambda_l(A_k) E^l=0, \quad {\rm for}~k\not=l\, .
\end{equation}
For $j=1,\ldots,r$, define 
\begin{equation} \label{KJPJ} K^j:=E^j+\theta E^j\quad\hbox{ and}\quad  P^j:=E^j-\theta E^j.\end{equation}

Denote by $\,I_0\,$ the $\,G$-invariant  complex structure  of $\,G/K$. On $\,\p \cong T_{eK} G/K$, it coincides 
with the adjoint action of the element $\,Z_0 \in Z(\k)\,$ given by 
\begin{equation}
\label{CENTER}
 \textstyle Z_0=S+\frac{1}{2} \sum_{j=1}^r K^j \,,
 \end{equation}
for some element   $\,S\in\m\,$  (see Lemma 2.4 in \cite{GeIa13}). 
The complex structure $\,I_0\,$ permutes the blocks of the decomposition (\ref{DECO}) of~$\,\p$.
Indeed, from the normalizations (\ref{NORMALIZ1}) and  (\ref{NORMALIZ2}), one sees that 
\begin{equation} \label{CPLX0}
I_0P^j =[Z_0,\,P^j]=A_j\qquad {\rm and} \qquad I_0A_j=[Z_0,\,A_j]=-P^j \,, 
\end{equation}
for $\,j=1,\ldots,r$. 
In particular $\,I_0\a=\bigoplus_{j=1}^r \p[\lambda_j].$
Moreover, one can easily check that 
\begin{equation}
\label{CPLXBIS}
\quad I_0\p[e_k+e_l]=\p[e_k-e_l]\,, \quad \quad \quad I_0 \p[e_j]=\p[e_j] \  \hbox{(non-tube case).}
 \end{equation}

For $\,a=\exp iH$, with $\,H \in \a$, define a $\,\C$-linear operator  $ \,F_a :\p^\C \to \p^\C\, $  by 
\begin{equation}
\label{EFFEDEF}
F_a := \pi_\# \circ \Ad_{a^{-1}}|_{\p^\C}\,,
 \end{equation}
 where $\, \pi_\# :\g^\C \to \p^\C\,$ be the linear projection  along $\,\k^\C$.
 One easily checks that  %$F_a$  is  $\Ad_K$-equivariant ???? and 
\begin{equation}\label{EFFE}F_aA =A \qquad {\rm and} \qquad F_aP^\alpha=\cos\alpha(H)P^\alpha\,, \end{equation}
for all $\,A\in\a\,$ and $\,P^\alpha\in \p[\alpha]$. In particular, for every $\,H \in \Omega $, the operator
$\,F_a\,$ is an isomorphism and  $\,F_a(\p)=\p$.

For $\,z \in G^\C/K^\C\,$ and  $\,Z\in \g^\C$, let 
\begin{equation}\label{TILDEFIELDS}
\textstyle
\widetilde Z_z:=\dds \exp(sZ) \cdot z \end{equation}
be the vector field induced by the holomorphic $\,G^\C$-action on $\,G^\C/K^\C$.
For $\,z=aK^\C\,$  on the slice of $\,\Xi\,$ and $\,Z\in\p^\C$,  one has \begin{equation}\label{TILDEVSHAT}\widetilde Z_z = a_*F_aZ ,\qquad 
a_*Z=\widetilde{F_a^{-1} Z}_z. \end{equation}

Denote by $\,\Xi'\,$ the $\,G$-invariant  subdomain of $\,\Xi\,$
defined by
\begin{equation}\label{REGULAR}
\Xi':=G \exp i\Omega' K^\C\,,
\end{equation}
where
$\,\Omega' :=\{ H \in \Omega  \ :\ \alpha(H)\not=0, ~\forall \alpha\in
\Sigma\}\,$
is the regular subset of $\,\Omega $.
Note that $\,\Omega' \,$ is dense in $\,\Omega \,$ and $\,\Xi'\,$ is dense in $\,\Xi$.

Later on, in the computation of the potentials of the various K\"ahler forms, we need the  identities   
contained in the next lemma. 

 \bigskip
\begin{lem}
\label{CURVESSTA} 
Fix   $\,z=aK^\C$, with $\,a= \exp iH\,$ and $\,H \in \Omega' $.  Decompose $\,Y\in \p$ as $\,Y=Y_\a + \textstyle\sum_\alpha
Q^\alpha\,$, where $\,Y_\a\in\a\,$ and $\,Q^\alpha= Y^\alpha- \theta Y^\alpha \in\p[\alpha]\,$ (see  Lemma \ref{BASIS}). Then
 \begin{itemize} 
 \item[(i)]   $\  \widetilde {iY}_z=
\dds \exp sC \exp i(H+sY_\a) K^\C\,,$ 
where %$\,C\,$ is the element of $\,\k\,$ defined by 
$$\,C:=- \textstyle  \sum_\alpha
\frac {\cos \alpha(H)}{\sin \alpha(H)}
  K^\alpha \,\qquad \hbox{and }  \qquad\,K^\alpha:=Y^{\alpha}+\theta Y^{\alpha}{\it .}$$

 \item[(ii)]     $\ \widetilde{iQ^\alpha}_z = a_*iF_aQ^\alpha=
   - \frac{\cos \alpha(H)}{\sin \alpha(H)}\widetilde{K^\alpha}_z$.
 
 \item[(iii)]     $\widetilde{K^\alpha}_z= -{\sin \alpha(H)} a_*iQ^\alpha$.

\end{itemize}
\end{lem}

\smallskip
\begin{proof} 
(i) By (\ref{TILDEVSHAT}) one has
 \begin{equation}\label{TILDEQ}\widetilde {iY}_z=a_*F_aiY=a_* \big (
 iY_\a+\textstyle \sum_\alpha \cos \alpha(H)  iQ^\alpha \big ) \,.\end{equation}
On the other hand, for $K^\alpha=Y^\alpha+\theta Y^\alpha$  and 
$\,C=\sum_\alpha c_\alpha K^\alpha$, Campbell-Hausdorff formula  yields
  $$\exp sC \exp i(H+sY_\a) K^\C=a\exp sAd_{a^{-1}}C\exp siY_\a  K^\C=$$
 $$ =a\exp(sAd_{a^{-1}}C+siY_\a+{s^2\over 2}[Ad_{a^{-1}}C,iY_\a]+\ldots)K^\C.$$
 %, \quad\hbox{ for $s\in\R$}
Differentiating the above expression at  $\,0$, one obtains
$$ 
\textstyle 
a_*\pi_\# (Ad_{a^{-1}}C+iY_\a)=a_* \big(iY_\a-\sum_\alpha c_\alpha\sin\alpha(H) iQ^\alpha
\big ).$$
Then the required identity follows by taking
 $ \,c_\alpha=- \frac {\cos \alpha(H)}{\sin \alpha(H)}.$ 

\sn
(ii) This is a  special case of (i).

\sn
(iii)  By setting $\,C=K^\alpha\,$ and 
$\,Y_\a=0\,$ in the  proof of (i), one obtains

\noindent
\smallskip
\centerline{$
\textstyle 
\widetilde {K^\alpha}_z=a_*\pi_\# Ad_{a^{-1}}K^\alpha=-a_* \sin\alpha(H) iQ^\alpha\,.$}
\end{proof}

%-------------------------------------------------------------------
%-----------------I----------------------------
%-------------------------------------------------------------------

\section{The complex structure $I$}
\label{I}
\bigskip
In this section we introduce a new $\,G$-invariant almost complex structure $\,I\,$ on $\,\Xi$. Its integrability will   be  settled  in Section \ref{PROOF}.
Eventually, $\,I$, $\,J=J_{ad}\,$ and $\,K:=IJ\,$ will be the three complex structures of our hyper-K\"ahler structure on~$\,\Xi$. 

%---------------------definition

\bigskip
\begin{defi}
\label{CPLX} For $\,gaK^\C\,$ in $\,\Xi\,$ and $\,Z \in \p^\C\,$ 
the $\,G$-invariant  (almost) complex structure $\,I\,$ is defined by
$$\,I g_*\widetilde Z_{aK^\C}= g_*\widetilde{ \overline {I_0Z}}
_{aK^\C}\,. $$
 \end{defi}
\nmedskip

We claim  that the above definition is well posed. Suppose that  $\,z=gaK^\C=g'a'K^\C\,$, for some $\,g,\,g'\in G\,$ and
$\,a,a'\in\exp i\Omega $,  and that $\,g_*\widetilde Z_{aK^\C}=g'_* \widetilde U_{a'K^\C}$, for some
$\,Z,\,U\in \p^\C$. This is equivalent to $\,g=g'wk\,$ and $\,a=w^{-1}a'w$, for some $\,w\in N_K(\a)\,$ and $\,z\in Z_K(a)\,$ (see \cite{KrSt05}, Prop.4.1), and $\,Ad_{wk}Z = U$. 
Then
$$ I g'_*\widetilde U_{a'K^\C}=g'_*\widetilde{\overline{I_0 U}}_{a'K^\C}= g'_* \widetilde{\overline{I_0 Ad_{wk}Z}}_{wa w^{-1}K^\C} =g'_*(wk)_* \widetilde{\overline{ I_0 Z}}_{a K^\C}=$$
 $$=g_*\widetilde{\overline{I_0 Z}}_{ a  K^\C}=  Ig_* \widetilde Z_{a K^\C},$$
as claimed.

By equations (\ref{TILDEVSHAT}) one   has  
\begin{equation}
\label{CPLXONSLICE}
 \,Ia_*Z=I\widetilde{F_a^{-1}Z}_z= \widetilde{ \overline {I_0F_a^{-1}Z}}_z=
 a_*F_aI_0F_a^{-1} \overline Z\,
\end{equation}
 for every $\,z=aK^\C\,$ on the slice of $\,\Xi$. 
 In particular,  for $\,\alpha \in \Sigma^+ \cup \{0\}\,$ and   $\,P^\alpha \in \p[\alpha]\,$ with
 $\,I_0P^\alpha \in \p[\beta]$, one has 
$$\textstyle I a_*P^\alpha= a_*\frac{\cos \beta(H)}{\cos \alpha(H)} I_0P^\alpha\,.$$

 From Definition \ref{CPLX} it is also clear that $\,I^2=-Id\,$ and $\,IJ=-JI$.
Then, by defining $\,K:=IJ$, one obtains a quaternionic (almost) complex structure
$\,(I,\,J,\,K)\,$ on $\,\Xi$.

%--------------------POTIP

\bigskip
\begin{prop}
\label{OLOMAP}
The $\,G$-equivariant projection $$p: \Xi \to G/K\,\quad \quad  gaK^\C \to gK\,,$$ is holomorphic with respect to 
the $\,G$-invariant  complex structures $\,I\,$ on $\,\Xi\,$ and
$\,I_0\,$ on $\,G/K$.
\end{prop}

\medskip
\begin{proof} Since $\,p\,$ is  $\,G$-equivariant,    it is sufficient  to consider its restriction to   the slice.
 Let $\,Z=X+iY\,$ be an element of $\,\p^\C$, with $\,X,Y\in\p$. We claim that the differential
 $\,p_*\colon T\Xi\to TG/K\,$ at  $\,z=aK^\C$, 
is given by $$p_*(\widetilde Z_z)=p_*(\widetilde{X}_z)+p_*(\widetilde{iY}_z) =X.$$
It is straightforward to check that  $\,p_*(\widetilde X_z)=X $. In order to verify that 
$\,p_*(\widetilde{ iY}_z)=0$, write $\,Y=Y_\a+\sum_\alpha Q^\alpha$, according to Lemma \ref{BASIS}. By  Lemma \ref{CURVESSTA}(i) one has 
$\,\widetilde{iY}_z =\dds\exp sC\exp siY_\a aK^\C$, for an appropriate element $\,C\in \k$.
Then from the definition of $\,p\,$ it follows that $\,p_*(\widetilde{ iY}_z)=0$.
%By the $G$-equivariance of $\pi$ it is enough to consider $z=aK^\C$ in the slice.
Now for $\,Z \in \p^\C\,$ one has 
$$ p_*(I\widetilde {Z}_z)=p_*(\widetilde {\overline{I_0Z}}_z)= 
p_*(\widetilde {I_0X}_z)- p_*(i\widetilde {I_0Y}_z)=I_0X =I_0p_*(\widetilde {Z}_z),$$
which concludes the proof of the statement.
 \end{proof}

%-------------------------------------------------------------------
%-----------------DEFORMATION---------------------------
%-------------------------------------------------------------------

\medskip
\section{The inverse of Dancer-Sz\"oke's deformation vs. the complex structure $I$}
\label{DEFORMATION}
\medskip
In this section we define a $\,G$-equivariant diffeomorphism $\,\psi\,$ 
 of the tangent bundle $\,TG/K\,$ with the property that  our complex structure $\,I\,$ is the pull-back via $\,\psi\,$ of the 
natural  complex structure of the holomorphic cotangent bundle $\,T^*G/K^{1,0} \cong T^*G/K \cong TG/K\,$
 (see also Rem.$\,$\ref{PULLSIMPL}). 
The map $\,\psi\,$ is the inverse of the diffeomorphism introduced in \cite{DaSz97}, Sect.$\,$4.
However, here $\,\psi\,$ and $\,I\,$ are expressed in a Lie theoretical fashion, a fact that 
 will be repeatedly exploited
in the sequel.

By identifying  the tangent bundle $TG/K$ with the homogeneous vector bundle $\,G\times_K\p$, the map $\,\psi\,$ is completely determined  by its restriction $\,\Psi \colon \p \to\p$, namely 
$\,\psi[g,X]=[g,\Psi(X)]$, for $\,g\in G\,$ and $\,X\in\p$. Note that $\,\psi\,$ maps every fiber into itself.

\medskip

 Let $\,Z_0\in Z(\k)\,$ be the element inducing the complex structure $\,I_0\,$ on $\,\p\,$
 and let $\, \pi_\# :\g^\C \to \p^\C\,$ be the linear projection  along $\,\k^\C$.

%----------------the deformation

\begin{lem} 
Let  $\,\Psi:\p \to \p\,$ be the map defined by 
$$\Psi(Y) := -I_0 \circ J \circ \pi_\# \circ Ad_{\exp iY}Z_0\,.$$
Then
\begin{itemize}
\smallskip
\item[(i)]  $\,\Psi\,$  is $\,\Ad_K$-equivariant,
\smallskip
\item[(ii)]  for $\,H=\sum_jt_jA_j\,$ in $\,\a\,$ one has 
$$\Psi(H)=\textstyle {1 \over 2}\sum^r_{j=1} \sin \lambda_j(H) A_j=\textstyle {1 \over 2}\sum^r_{j=1} \sin(2t_j)A_j\,.$$
\end{itemize}
\end{lem} 
 
\begin{proof} (i) Since $\,Z_0\,$ lies in the center of $\,\k$, for every $\,k \in K\,$ one has 
$\,Ad_{\exp  i \Ad_k Y}Z_0=Ad_{k\exp  i  Y}Z_0=Ad_k \circ Ad_{\exp  i  Y}Z_0$. Now the statement follows 
from the $\,\Ad_K$-equivariance of all  the remaining maps in the composition
defining $\Psi$.
% and   $\,J:\p^\C \to \p^\C\,$  and $\, \pi_\#:\g^\C \to \p^\C\,$.

\sn
(ii) One has
$$Ad_{\exp  iH}Z_0=e^{ad_{ iH}}Z_0=\cos ad_H Z_0 +i\sin ad_H Z_0=$$
$$=\textstyle \sum_{n\ge 0}{{(-1)^n}\over {(2n)!}}ad^{2n}_{H}Z_0+ i
\textstyle \sum_{n\ge 0}{{(-1)^n}\over {(2n+1)!}}ad^{2n+1}_{H}Z_0.$$
Lemma \ref{BASIS} and relations (\ref{CENTER}) and (\ref{CPLX0}),  imply that
$$
\textstyle
ad_{H}^{2n}Z_0= {1\over 2} \textstyle \sum_j \lambda_j^{2n}(H) K^{j},\quad ad_{H}^{2n+1}Z_0=
 {1\over 2}\sum_j \lambda_j^{2n+1}(H) P^j. 
$$
It follows that  the $\,\p^\C$-component of $\,Ad_{\exp  iH}Z_0\,$ is given by
$$
\textstyle
\pi_\#Ad_{\exp  iH}Z_0=   {i\over 2} \textstyle \sum_j   \sin \lambda_j(H) P^j= {i\over 2}\sum_j \sin (2t_j)P^j\,,$$
and
$$-I_0 \circ J \circ \pi_\# \circ Ad_{\exp iH}Z_0 =$$
$$
\textstyle
=-I_0 \circ J \big( {i\over 2}
\sum_j \sin \lambda_j(H)P^j
\big)= {1\over 2} I_0 \sum_j \sin \lambda_j(H)P_j= {1 \over 2}\sum_j \sin (2t_j)A_j,$$
as claimed.
\end{proof}

%--------------------PATH

For $\,a=\exp iH$, with $\,H\in\Omega $, consider the $\,\C$-linear map $\,E_a\colon\p^\C\to \p^\C\,$ uniquely defined  by  $$\,E_a:= \pi_\#\circ  \widetilde E_a|_{\p^\C}, $$  where $\,\widetilde E_a\colon \g^\C\to\g^\C\,$ is given by $ \,\widetilde E_a= \sum_{n \ge 0} \frac{(-1)^{n}}{(n+1)!} \ad^n_{iH}\,$,
and has  the property that (cf. \cite{Var84}, Thm. 2.14.3, p. 108) $\,(\exp_*)_{iH}=(\exp iH)_* \widetilde E_a$.
One can verify that 
\begin{equation} \label{E}E_aA=A,\qquad\qquad E_aP^\alpha= \textstyle \frac {\sin\alpha(H)}{\alpha(H)}P^\alpha,   \end{equation}
for all $A\in\a$ and $P^\alpha\in\p[\alpha]$, with $\alpha\in\Sigma$.

\bigskip
\begin{lem}
\label{PATH} Fix $\,a=\exp iH$, with $\,H \in \Omega' $, and 
 $\,P^\alpha \in \p[\alpha]$. 
 %with $\,P^\alpha= X^{\alpha}- \theta  X^{\alpha},$ let  $\,K^{\alpha}= X^{\alpha}+ \theta  X^{\alpha}$.
Then 
$$ \textstyle \dds  \Ad_{\exp i(H+sP^\alpha)}Z_0=
\dds  \Ad_{\exp - s \frac{K^\alpha}{\alpha(H)} \exp  iH}Z_0 \,.$$ 
\end{lem}

\begin{proof} 
Since $\,\Ad _{G^\C}Z_0 \cong G^\C/K^\C\,$ the statement can be reformulated as
$$ \textstyle \dds \exp i(H+sP^\alpha)K^\C=
 \dds \exp -s \textstyle \frac{K^\alpha}{\alpha(H)} \exp  iH K^\C \,.$$
By the definition of $\,E_a$, the  left-hand side is  $\,a_*E_a iP^\alpha$; likewise  the right-hand side is 
$$
\textstyle
-a_*\pi_\#Ad_{a^{-1}}{1\over\alpha(H)}K^\alpha=   a_*{{\sin\alpha(H)}\over{\alpha(H)}} iP^\alpha= a_*E_a iP^\alpha.$$
%Then just observe  that $$\,-\pi_\# \circ \Ad_{a^{-1}} \textstyle \frac{K^\alpha}{\alpha(H)}=  \textstyle i\frac{\sin \alpha(H)}{\alpha(H)} P^\alpha =E_a iP^\alpha\,.$$
\end{proof}

%--------------------DIFFERENTIAL------------
\nbigskip

Fix  $\,H \in \a$.  Identify as usual  $\,T_H\p$ and $T_{\Psi(H)}\p$, the tangent spaces to $\,\p\,$ at $\,H\,$ and at $\,\Psi(H)$,  with $\,\p$.
Consider the differential $\,(\Psi_*)_H:\p \to \p\,$ of  $\,\Psi\,$ at $\,H$.

\bigskip
\begin{lem} {\rm (cf. \cite{DaSz97}, Lemma 2.4)}
Fix $\,a=\exp iH$, with $\,H = \sum_jt_jA_j$ in $\Omega'$. Then
\label{DIFFERENTIAL}
\begin{itemize}
\smallskip
\item[(i)]  $\,(\Psi_*)_H A_j = \cos \lambda_j(H)A_j,$ for all $\,j=1, \dots,r$,

%$\,(\Psi_*)_H A = {1\over 2} \sum_j\cos \lambda_j(H)\,\lambda_j(A)A_j,$ for every $A\in \a$;
\smallskip
\item[(ii)]  
$\,(\Psi_*)_H P^\alpha= \frac {\alpha (\Psi(H))}{\alpha (H)} P^\alpha$, for all $\,P^\alpha \in \p
[\alpha]\,$ with $\, \alpha \in \Sigma^+$.
\end{itemize}
In particular $\, (\Psi_*)_H\,$ is self-adjoint with respect to the Killing form $\,B$.
\end{lem}
 
 \medskip
\begin{proof} 
Part (i) follows from the definition of $\,\Psi$. 

\noindent
(ii) By Lemma \ref{PATH} and the $\,\Ad_K$-equivariance of $\,\Psi$ one has 
$$\textstyle 
(\Psi_*)_H P^\alpha= \dds \Psi (H+sP^\alpha ) =-I_0 \circ J
\circ \pi_\# \big(\dds  \Ad_{\exp  i(H+sP^\alpha)}Z_0 \big )=$$
$$\textstyle 
=-I_0 \circ J \circ \pi_\# \big (\dds \Ad_{\exp -s \textstyle \frac{K^\alpha}{\alpha(H)} \exp iH}Z_0
\big)  
 = \dds \Ad_{\exp -s \textstyle \frac{K^\alpha}{\alpha(H)}}  \Psi(H) $$
$$\textstyle 
= -[\textstyle \frac{K^\alpha}{\alpha(H)},\,\Psi (H)]=
 \frac {\alpha (\Psi(H))}{\alpha (H)} P^\alpha\,.$$
\end{proof}

Extend $\,\C$-linearly   $\,(\Psi_*)_H\colon \p^\C\to\p^\C$, and 
define  a $\,G$-invariant, real analytic map $\,L:\exp i\Omega' \to {\rm GL}(\p^\C)$ by 
$$a \to L_{a}:=I_0F_a^{-1}(\Psi_*)_HE_a^{-1}\,.$$

  %----------------------------------------------                                      

\bn
\begin{lem}
\label{STRUTTURACOMPLESSA}
Given $\,a \in \exp i\Omega' \,$ one has $\,L_a = F_aI_0F_a^{-1}$. 
 In particular  $\,L\,$  extends real-analytically  to $\,\exp i\Omega\,$ and 
 $$Ia_*Z= a_*L_a \overline Z\,,$$ for every $\,Z \in \p^\C$.
\end{lem}

\medskip
\begin{proof} Since the maps $\,F_a$, $\,E_a\,$ and $\,(\Psi_*)_H\,$ are $\,\C$-linear and commute, the statement 
of the lemma is equivalent to 
 \begin{equation}
 \label{RELAZIONE2}
 (\Psi_*)_HE_a^{-1}=-I_0F_aI_0
 \end{equation} 
 on $\,\p$. Recall that $\,I_0\,$ permutes the blocks of decomposition (\ref{DECO}), namely
$$I_0A_j \in \p[2e_j],\qquad I_0 \p[e_j]=\p[e_j],\qquad
 I_0\p[e_k+e_l]=\p[e_k-e_l]\,.$$
As $\,\lambda_j=2e_j$, by (\ref{CPLX0}) and Lemma \ref{DIFFERENTIAL},   one easily verifies that
$$(\Psi_*)_HE_a^{-1}A_j=\cos\lambda_j(H) A_j=-I_0F_aI_0 A_j$$ and  that 
$$\textstyle (\Psi_*)_HE_a^{-1}I_0A_j={{\lambda_j(\Psi(H))}\over{\sin\lambda_j(H)}}I_0A_j=I_0A_j=-I_0F_aI_0 \,I_0 A_j.$$
If $\,I_0\p[\alpha]=\p[\beta]$, with $\,\alpha,\beta\not=0$,
then we have that  $\,(\Psi_*)_HE_a^{-1}P^\alpha=-I_0F_aI_0 P^\alpha\,$
if and only if the following identity holds true
 \begin{equation}\label{TRIGO} \alpha(\Psi(H))=\sin\alpha(H)\cos\beta(H).\end{equation} 
For  $\,\alpha=\beta=e_k$, equation (\ref{TRIGO}) becomes $\,{1\over 2}\sin 2t_k=\sin t_k\cos t_k$, which is obviously verified.
For $\,\alpha=e_k\pm e_l\,$ and $\,\beta=e_k\mp e_l$, equation (\ref{TRIGO}) becomes
$$\textstyle {1\over 2} \left(\sin 2t_k \pm \sin 2t_l\right)=\sin(t_k\pm t_l)\cos (t_k\mp t_l),$$
which can be easily checked.
\end{proof}

\begin{remark}
\label{PULLI}
By using  the identity 
$$L_a=E_a(\Psi_*)_H^{-1}I_0(\Psi_*)_HE_a^{-1}\  $$
one can also verify that the complex structure $\,I\,$  is the pull-back via $\, \psi\,$ of the natural complex structure on the 
holomorphic cotangent bundle $\,T^*G/K^{1,0}\cong T^*G/K \cong TG/K\,$ of $\,G/K\,$ (see also Rem. \ref{PULLSIMPL}) .
\end{remark}

%-------------------------------------------------------------------
%-----------------Existence----------------------------
%-------------------------------------------------------------------

\medskip
\section{The hyper-k\"ahler structure}
\label{STRUCTURE}
\bigskip

In this section we introduce three $\,G$-invariant differential 2-forms $\,\omega_I$, $\,\omega_J\,$
and $\,\omega_K\,$ on the crown domain $\,\Xi\,$ in $\,G^\C/K^\C$ and study their basic properties. 
The forms $\,\omega_I\,$
and $\,\omega_K\,$ are restrictions of $\,G^\C$-invariant forms on $\,G^\C/K^\C\,$ and therefore closed;
$\,\omega_J\,$ will be shown to be closed in Proposition \ref{POTENTIAL}. 
The forms $\,\omega_I$, $\,\omega_J\,$
and $\,\omega_K\,$  are invariant under the  (almost) complex structures
$\,I$, $\,J\,$ and $\,K$, respectively.  
   Eventually, they
will be   the three K\"ahler forms of our hyper-K\"ahler structure.     

 %---------------------definition

\bigskip
\begin{defi}
\label{FORMS}
For $\,g \cdot z\in \Xi$, with $\,z=aK^\C$, and $\,Z \in \p^\C\,$ 
define  $\,G$-invariant real-analytic forms by
%(definirle solo sulla slice e dire che sono invariati??)
  
\begin{itemize}
\smallskip
\item[]
 $\omega_I(g_*\widetilde Z_{z}, \,g_*\widetilde W_{z}):=\ \ \, {\rm Re}B (I_0F_{a} Z,\,F_{a}  W)\,$,
 
\smallskip
\item[]
  $\omega_J(g_*\widetilde Z_{z}, \,g_*\widetilde W_{z}):=\omega_I(J\widetilde Z_z, I \widetilde W_z)=-{\rm Im}B(I_0F_a  Z,  F_aI_0 \overline W)$, 
   
\smallskip
\item[]
 $\omega_K(g_*\widetilde Z_{z}, \,g_*\widetilde W_{z}):=-{\rm Im}B (I_0F_{a} 
Z,\,F_{a} 
 W) $, 
   \end{itemize}
 where  $F_a=\pi_\# Ad_{a^{-1}}|_{\p^\C}$ $(\,$see $($\ref{EFFEDEF}\,$)\,)$.
   \end{defi}

  \mn
  
  We claim that  the forms $\,\omega_I$, $\,\omega_J\,$ and 
 $\,\omega_K\,$  are   well   defined. Indeed,
 assume that  $\,gaK^\C=g'a'K^\C$, for some  $\,g,\,g'\in G$ and $\,a,\,a'\in\exp i\Omega$, and that 
 $\,g_*\widetilde Z_{aK^\C}=h_* \widetilde U_{a'K^\C}\,$ and $\,g_*\widetilde W_{aK^\C}=g'_* \widetilde V_{a'K^\C}$, for some
 $\,Z,\,U,\, W,\, V\in \p^\C$.  
This  is equivalent to $\,g=g'wk\,$ and $\,a=w^{-1}a'w$, for some $\,w\in N_K(\a)\,$ and $\,z\in Z_K(a)\,$
 (see \cite{KrSt05}, Prop.4.1), and in addition $\,  Ad_{wk}Z = U \,$ and $\,  Ad_{wk}W= V $.
Then, from the definition of the operator $\,F_a$, the $\,\Ad_K$-equivariance of $\,\pi_\#\,$ and the $\,\Ad_K$-invariance of
$\,B$, it follows that 
  $$B (I_0F_{a'} U,\,F_{a'} V)  
   =B (I_0 \pi_\#Ad_{wa^{-1} w^{-1}} Ad_{wk}Z, \pi_\#Ad_{wa^{-1} w^{-1}} Ad_{wk}W)=$$
   $$=B(I_0 F_a Z,F_aW).$$
As a result, 
 $$\omega_I(g'_*\widetilde U_{a'K^\C},g'_*\widetilde V_{a'K^\C})={\rm Re} B (I_0F_{a'} U,\,F_{a'} V) =  {\rm Re}B(I_0 F_a Z,F_aW)=$$
 $$=\omega_I(g_* \widetilde Z_{aK^\C},g_* \widetilde W_{aK^\C}),$$
which says that  $\,\omega_I\,$ is well defined.
A similar reasoning applies to  $\,\omega_K$.

The form   $\,\omega_J\,$ is well defined, since the complex structure $\,I\,$ and the form $\,\omega_I\,$ are.
For $\,Z,W\in\p^\C$, one has 
$$\omega_J(g_*\widetilde Z_{z},\,g_*\widetilde W_{z})=
\omega_I(Jg_*\widetilde Z_{z},\,Ig_*\widetilde W_{z})=
\omega_I(g_*\widetilde{iZ}_{z},\,g_*\widetilde{\overline{ I_0 W}}_{z})=$$
$$={\rm Re}B( I_0 F_a iZ,F_aI_0 \overline W)
=-{\rm Im}B(I_0 F_a  Z, F_aI_0 \overline W).$$

 %-------------------FORMS

 \bn
 \begin{lem}
 \label{FORMS2}
 \item
\begin{itemize}
\item[(i)]
On $\,\Xi'\,$  one has
    $$\omega_J(g_*\widetilde Z_{aK^\C}, \,g_*\widetilde W_{aK^\C})=-{\rm Im}B \big ( 
  Z,\,(\Psi_*)_H  E_{a}^{-1} F_{a} \overline W \big ) .$$

\smallskip
\item[(ii)] 
 For $\,gaK^\C \in \Xi\,$ and 
 $\,Z,W\in\p^\C$, one has

\begin{itemize}

\smallskip
\item[]
$\omega_I((ga)_*Z , (ga)_*W )\ =\ \ \,{\rm Re} B (I_0Z ,\,  W );$

\smallskip
\item[]
 $\omega_J((ga)_*Z , (ga)_*W )\  =-{\rm Im} B 
(I_0 Z,\, F_a I_0F_{a}^{-1}  \overline W)$;
%=-{\rm Im} B(Z,F_a^{-1}(\Psi_*)_H  E_{a}^{-1} \overline W)  $; 

\smallskip
\item[]
$\omega_K((ga)_*Z , (ga)_*W )\ =-{\rm Im} B (I_0Z ,\,  W ) $.
\end{itemize}
\smallskip
\item[(iii)] The forms $\,\omega_I\,$ and $\,\omega_K\,$ are locally
$\,G^\C$-invariant %$($cf. Lemmas \ref{OLOSIMPLECTIC} and \ref{INITIAL}$)$ 
and are closed. 
\end{itemize}
\end{lem}
 
 \medskip
 \begin{proof}
 (i) 
By (\ref{RELAZIONE2}),    the  quantity on the right-hand side equals 

$$- {\rm Im}B(   Z,  (\Psi_*)_H E_a^{-1}  F_a \overline W) =- {\rm Im}B(F_a  Z, \,  (\Psi_*)_H E_a^{-1}  \overline W)=$$
$${\rm Im}B(F_a  Z,\, I_0 F_aI_0 \overline W)=-{\rm Im}B(I_0F_a  Z,\, F_aI_0 \overline W)\,,$$
 as claimed.

%$I_0 F_aI_0=-(\Psi_*)_H E_a^{-1}$.  

\bn 
(ii)
%Fix a point $z=aK^\C$ in  $\,\Xi''\,$.  
One has 
$\,a_*Z=\widetilde{F_a^{-1}Z}_{aK^\C}$. %$\,\widetilde Z_z= (\widehat{F_{ga} Z})_z$. 
The statements about $\,\omega_I\,$ and $\,\omega_K\,$ are immediate.
For  $\,\omega_J$, one has 
$$\omega_J((ga)_*Z, (ga)_*W )=\omega_I(J(ga)_*Z,\,I(ga)_*W )=\omega_I((ga)_*iZ,\,(ga)_* L_a \overline W)=$$
$$\omega_I((ga)_*iZ,\,(ga)_* F_aI_0 F_a^{-1} \overline W)= {\rm Re}\, iB(I_0 Z,F_aI_0 F_a^{-1} W)=
-{\rm Im}B(I_0Z,F_aI_0 F_a^{-1}\overline W).$$

\bn
(iii) We  first show  that  the complex form $\,\omega_I-i \omega_K\,$ is locally 
$\,G^\C$-invariant. To   this aim, consider the map $\,\p \times i\p \times K^\C \to G^\C$, given by 
$\,(X,\,iY,\,k) \to \exp X \exp iY k$, which sends a neighborhood of 
$\,(0,\,0,\,e)\,$ in $\,\p \times i\p \times K^\C\,$ onto a neighborhood
of $\,e\,$ in $\,G^\C$.  For every  $\,Y\,\in \,\p\,$ one can write  $\,\exp iY = h \exp iH 
h^{-1}$, for some
$\,h \in K\,$ and $\, H \in \a$. % (anche qui probabilmente serve restringersi a $\Omega^+$)  
Then,  due to the $\,K^\C$-equivariance of $I_0$, 
the $\,G^\C$-invariance of $\,B\,$ and the formulas proved in (ii), 
for $\,\exp X \exp iY k\,\in \exp \p\exp i\p K^\C\,$
one has
$$(\omega_J-i\omega_K)((\exp X \exp iY k)_*Z, \,(\exp X \exp iY k)_*W)=$$
$$=(\omega_J-i\omega_K)((\exp X h \exp iH h^{-1} k)_*Z, \,(\exp X 
h \exp iH h^{-1} k)_*W)=$$
$$=B (I_0\Ad_{h^{-1} k}Z ,\,  \Ad_{h^{-1} k}W )=B (I_0Z ,\, W )=\,(\omega_J-i\omega_K)(Z, \, W)\, .
$$
 By letting   $\,\exp X \exp iY k\,$ vary  in 
a neighborhood of $\,e\,$ in $\,G^\C$, one concludes that
the forms 
are locally  $\,G^\C$-invariant near $eK^\C$.

Since $\,\omega_I\,$ and $\,\omega_K\,$ are real-analytic, they are
 restrictions of 
 $\,G^\C$-invariant forms on $\,G^\C/K^\C$.
 In order to prove that they are closed,
 we adapt  
 the proof in \cite{Wol84}, Thm. 8.5.6, p. 250, to  our complex  setting. 
 
Let $\,\omega \,$ be an arbitrary $\,G^\C$-invariant 2-form on $\,G^\C/K^\C$.  
A similar argument as in \cite{Wol84}, Lemma 8.5.5, shows that $\,\omega\, $ is
closed if and only if so is its  pull-back  $\,\pi^*\omega\,$  to $\,G^\C$.
Since $\,\pi\,$ is $\,G^\C$-equivariant, the form $\,\pi^*\omega\,$ is $\,G^\C$-invariant.
Given left invariant vector fields $\,\widehat X$, $\,\widehat Y$, $\,\widehat W$
on $\,G^\C\,$ such that at least one of them lies in the kernel $\,\k^\C\,$ of the differential
$\,\pi_*\,$ at $\,e$, one has 
$$d\pi^*\omega(\widehat X,\,\widehat Y,\,\widehat W)\equiv 
d\pi^*\omega(\widehat X_e,\,\widehat Y_e,\,\widehat W_e)= 
\pi^*d\omega(X,\,Y,\,W)= $$
$$=d\omega(\pi_*X,\,\pi_*Y,\,\pi_*W) =0\,.$$
Hence we may assume that $\,X,\,Y,\,W\in \p^\C$.
From Cartan's formula for the external derivation and 
the $\,G^\C$-invariance of $\,\pi^*\omega\,$ one has
$$d\pi^*\omega(\widehat X,\,\widehat Y,\,\widehat W)= 
\pi^*\omega(\widehat X,\,\widehat {[ Y,\,W]})-
\pi^*\omega(\widehat Y,\,\widehat {[X,\, W]})+
\pi^*\omega(\widehat W,\,\widehat {[X,\,Y]})\,,$$
which vanishes due to the inclusion $\,[\p^\C,\,\p^\C] \subset \k^\C$.
\end{proof}

%------------------PROPERTIFORMS

\bigskip
\begin{lem}
\label{PROPERTIFORMS}
The form $\,\omega_I\,$ is $\,I$-invariant, namely 
  $$\omega_I(\,\cdot\,,\,\cdot\,)=\omega_I(I \,\cdot\,,I\,\cdot\,) \,.$$
  Likewise,
  $\,\omega_J\,$ is $\,J$-invariant and $\,\omega_K\,$ is $\,K$-invariant.
\end{lem}
 
 \sn
  \begin{proof} By the $\,G$-invariance of the forms, it is sufficient 
 to prove the statements for $\,z=aK^\C\,$ in the slice $\,\exp i\Omega K^\C$.\pn
One has 
$$\omega_I(Ia_*Z,\,Ia_*W)=
\omega_I(a_*L_a \overline Z,\, a_* L_a\overline W)=
{\rm Re}B (I_0L_a \overline Z,\, L_a \overline W)=$$
$$-{\rm Re}B (F_a^{-1}(\Psi_*)_H E_a^{-1} \overline Z,L_a \overline W)=
-{\rm Re}B ( \overline Z,\,F_a^{-1}(\Psi_*)_H E_a^{-1}L_a  \overline W)=$$
$$-{\rm Re}B (I_0 \overline Z,\,L^2_a \overline W)=
{\rm Re}B (I_0 \overline Z, \, \overline W)=
{\rm Re}B (I_0 Z,\, W)=
\omega_I(a_*Z,\,a_*W)\,.$$
Similarly, one obtains the $\,J$-invariance of $\,\omega_J\,$ and 
 the $\,K$-invariance of $\,\omega_K$.
\end{proof} 

%-------------------------------------------------------------------
%-----------------POTMOMENT----------------------------
%-------------------------------------------------------------------

\bigskip
\section{An invariant potential for $\,\omega_J\,$ and the associated moment map }
\label{POTMOMENT}

\bigskip
In this section we exhibit a $\,G$-invariant potential for $\,\omega_J\,$,
i.e. a $\,G$-invariant, smooth  function
$\,\rho_J\,$ such that $\, 2i \partial  \bar \partial_J \rho_J= \omega_J$.
The fact that $\,\rho_J\,$  is   $\,J$-strictly plurisubharmonic implies
that   $\,\omega_J\,$ is a  K\"ahler form with respect to $\,J$.
We prove that $\, 2i \partial  \bar \partial_J \rho_J= \omega_J\,$
by applying   moment map techniques, namely 
  we use the  following  reformulation of Lemma 7.1 in
[HeGe07].

\medskip
Let $\,{\mathcal G}\,$ be a real Lie group acting  by holomorphic transformations on
 a  manifold  $\,M\,$    with a complex
structure $\,{\mathcal I}\,$. For $\,X \in Lie({\mathcal G})$, denote by 
$\,\widetilde X\,$ the vector field on $\,M\,$  induced by the $\,{\mathcal G}$-action,
 namely $\,\widetilde X_z:=\dds \exp sX \cdot z$.
Let $\,\rho:M \to \R\,$ be a $\,G$-invariant, smooth 
$\,{\mathcal I}$-strictly plurisubharmonic function. Set 
$\,d^c_{\mathcal I}\rho:=d\rho \circ {\mathcal I}$, so that 
$\,2i \partial  \bar \partial_{\mathcal I}\rho=-dd^c_{\mathcal I}\rho$. Then $\,-dd^c_{\mathcal I}\rho\,$ is 
a $\,{\mathcal G}$-invariant K\"ahler form with respect to $\,{\mathcal I}$ and the map 
$\,\mu:M \to Lie({\mathcal G})^*$, defined by 
$$ \mu(z)(X)=d^c_{\mathcal I}\rho(\widetilde X_z)\,,$$
for $\,X\in Lie({\mathcal G})$,  is a moment map. It  is referred to as the moment map associated
with $\,\rho$.

In order to compute  the $\,G$-invariant form $\,-dd^c_{\mathcal I}\rho$, also in the case when 
$\,\rho\,$ is not  known to be $\,{\mathcal I}$-strictly plurisubharmonic (as we do in Section \ref{OMEGAI}), %and \ref{OMEGA0}
one can apply the following lemma.

 %-----------DIFFER--------------------------
 
 \sn
\begin{lem}
\label{DIFFER} Let $\,\rho\colon M \to \R\,$ be a smooth $\,{\mathcal G}$-invariant function.  
 For $\,X\in Lie({\mathcal G})$, define  $\,\mu^X: M \to \R\,$ by 
 $\,\mu^X(z)=d^c_{\mathcal I}\rho(\widetilde X_z)\,.$
Then
 $$\,d\mu^X= -\iota_{\widetilde X}dd^c_{\mathcal I}\rho\,.$$ %(CONTROLLARE). 
 \end{lem}

\medskip
 \begin{proof} The same proof as the one of  Lemma 7.1 in \cite{HeSc07}
 applies to our situation. Indeed their argument needs
  neither  the compactness of $\,{\mathcal G}\,$ nor the plurisubharmonicity of $\,\rho$. 
   \end{proof}

%---------------------POTENTIAL
\bigskip
\begin{prop}
\label{POTENTIAL} 
Let  $\,z=gaK^\C$, with $\,a= \exp iH$, be an element in $\,\Xi$.
The $\,G$-invariant function $\,\rho_J:\Xi \to \R\,$ defined by 
$$\,\rho_J(gaK^\C):=  -\textstyle \frac{1}{4}\sum_{j=1}^r \cos 
\lambda_j(H) B(A_j,A_j)\,,$$
is a strictly plurisubharmonic potential for $\,\omega_J$. 
 The associated moment map $\,\mu_J:~\Xi \to \g^*\,$  is given by 
$$\,
\mu_J(gaK^\C)(X) = B \big (\Ad_{g^{-1}} (X) ,\,\Psi(H) \big )\,,$$
 for $\,X\in\g$.
%In particular,  the form $\,\omega_J=-dd_J^c \rho_J$ is closed.
\end{prop}

\bigskip
\noindent

Note that $\,B(A_1,\,A_1)= \dots = B(A_r,\,A_r)\,$ is a constant depending only on the symmetric space
$\,G/K$.

\bigskip
\begin{proof} 
The map $\, \Omega \to \R,$ given by $\,H \to \rho_J(\exp iH K^\C)$, is strictly 
convex. Then Theorem 10
in \cite{BHH03} implies that $\,\rho_J\,$ is strictly plurisubharmonic. 

The form $\,-dd^c_J\rho_J\,$ is computed by applying Lemma \ref{DIFFER}. For this we 
first determine $\,\mu^X_J(z):= d^c_J\rho_J(\widetilde X_z)$. In particular we 
obtain the moment map $\,\mu_J\,$ associated with $\,\rho_J$.

Fix $\,z=aK^\C$, with $\,a= \exp iH\,$ and $\,H \in \Omega'\,$ (see (\ref{REGULAR})).
Start with $\,X \in \p\,$ and 
 write  $\,X_\a={1\over 2} \sum_j \lambda_j(X_\a) A_j\,$ for the $\,\a$-component of $\,X$. %$\,X_\a=  \sum_{j=1}^r \tau_j A_j$.
By  Lemma \ref{CURVESSTA}(i) and the $\,G$-invariance of $\,\rho_J$, 
one has 
$$\textstyle
\mu_J^X(z)=d_J^c \rho_J (\widetilde X_z)= d \rho_J (\widetilde {iX}_z)=\dds \rho_J (\exp i(H+sX_\a) K^\C)=$$
$$
= \textstyle -\frac{1}{4}\dds    \textstyle \sum_{j=1}^r \cos \lambda_j(H+sX_\a)B(A_j,\,A_j)= $$
$$= \textstyle \frac{1}{4} \textstyle \sum_{j=1}^r  \sin \lambda_j(H) \lambda_j(X_\a)B(A_j,\,A_j)
= B \big (X_\a,\,\Psi (H) \big )= 
B \big (X,\,\Psi (H) \big ) \,.$$
The $G$-invariance of $\,\rho_J\,$ and of the complex structure $\,J\,$   
then implies %the $G$-equivariance of  $d^c_J\rho_J$, namely
$$\mu_J^X(g \cdot z)=d^c_J\rho_J(\widetilde X_{g\cdot z})=d^c_J\rho_J( \widetilde{{\Ad_{g^{-1}}X}_{z}})=
B \big (\Ad_{g^{-1}} X,\,\Psi (H) \big ),$$
for all $\,g\in G\,$ and  $\,X \in \p$.
In order to show that such an identity holds true also for
 $\,X \in \k$, write $\,X=M + \sum_{\alpha}  K^\alpha$, 
  with $\, M \in \m\,$ and $\, K^\alpha=X^{\alpha}+\theta X^{\alpha}$. Set
  $\, P^\alpha=X^{\alpha}-\theta X^{\alpha}$. One has  
%$$\,\mu_J(z)(C)=
$$ \textstyle d_J^c \rho_J (\widetilde X_z)= \dds \rho_J (\exp siXa K^\C)=$$
$$\textstyle \dds \rho_J \big (a \exp si \Ad_{a^{-1}}XK^\C \big )=
\dds \rho_J \big (a \exp \big (s \textstyle \sum_\alpha
\sin \alpha(H) P^\alpha \big ) K^\C\big)=$$
$$ \textstyle \dds \rho_J \big (\exp \big (s \textstyle\sum_\alpha
 \frac{\sin \alpha(H)}{\cos \alpha(H)}P^\alpha \big ) aK^\C \big )=0\,,$$
where the last equality follows from the  $\,G$-invariance of $\,\rho_J$.
Since $\,\a \perp_B \k\,$, it follows  that
 $\,d_J^c\rho_J (\widetilde X_z)=B \big (X,\,\Psi(H) \big )$, as wished. 
 
Next we need to show that $\,-dd_J^c\rho_J(\widetilde Z_z, \widetilde W_z)=\omega_J(\widetilde Z_z, \widetilde W_z),$ 
for all $\,Z ,~W \in\p^\C$.
Since  both forms are $\,J$-invariant, 
%$\,-dd_J^c \rho_J(J \,\cdot\,,\,J \,\cdot\,)=-dd_J^c \rho_J(\,\cdot\,,\,\,\cdot\,)$,
it is enough to  show that
 $$\,-dd_J^c \rho_J(\widetilde X_z,\,\widetilde W_z)
 = \omega_J(\widetilde X_z,\,\widetilde W_z)\,$$
  for all $\,X\in \p\,$
and $\,W\in \p^\C$. Moreover, since both forms are $\,G$-invariant %by the $G$-invariance of $\rho_J$
 it is sufficient to prove the above identity at points $\,z=aK^\C$, with $\,a= \exp iH\,$ and $\,H \in \Omega'\,$ (see (\ref{REGULAR})).
Write $\,W=U+iV$, with 
 $\,U,\, V \in \p$.  Then by Lemma~\ref{DIFFER}, one has 
 $$-dd^c \rho_J (\widetilde{X}_z,\widetilde W_z)=\,d\mu_J^X(\widetilde W_z) = d\mu_J^X(\widetilde U_z)+d\mu_J^X(\widetilde {iV_z)} .$$ 
 
 \sn
The first   summand on the right-hand side   is zero:
$$\textstyle d\mu_J^X(\widetilde U_z) =  %=-d B \big (\Ad_{g^{-1}} X,\,\Psi (H) \big )(\widetilde U_z) $$
 % =\,d\mu_J^X(\widetilde {U}_z)=\dds \mu_J^X(\exp sU aK^\C)=$$
   \dds B \big (\Ad_{\exp -sU} X,\,\Psi (H) \big )=
  \dds B \big (X-s[U,X],\,\Psi (H) \big )=0\,,$$
   since $\,[U,X] \in \k\,$ and $\,\k \perp_B \a$.   

\sn 
For the second summand,   write $\,V=V_\a + \sum_\alpha
 Q^\alpha$, where $V_\a\in\a$ and $Q^\alpha=V^\alpha-\theta V^\alpha\in\p[\alpha]$.
From Lemma \ref{CURVESSTA}(i), we get
  $$\,\textstyle d\mu_J^X(\widetilde {iV}_z)= \dds \mu_J^X(\exp sC \exp i(H +sV_\a) K^\C)=$$
  %$$-dd^c \rho_J (\widetilde{X}_z,\widetilde{iV}_z)= -\dds B(Ad_{\exp -sC}X,\,\Psi ({H +sV_\a})) =$$
  $$=  \textstyle \dds B \big (X-s[C,\,X],\,\Psi ({H +sV_\a}) \big )=  
  -B \big ([C,\,X],\,\Psi ({H})\, \big ) +B \big (X,\,(\Psi_*)_H V_\a \big ) =$$
   $$= \textstyle B \big (X,\,[C,\,\Psi ({H})]\, \big )+B \big (X,\,(\Psi_*)_H V_\a \big )  \,, $$
   where
   $$\,C = -\textstyle \sum_\alpha \frac {\cos \alpha(H)}{\sin \alpha(H)}
  K^\alpha\,\quad\hbox{and}\quad  K^\alpha=V^\alpha+\theta V^\alpha.$$
Hence  
 $$\,[C,\,\Psi(H)] = \textstyle 
 \sum_\alpha \frac {\alpha (\Psi(H))\cos \alpha(H)}{\sin \alpha(H)}
  Q^\alpha\,.$$
In addition, by (\ref{EFFE}), (\ref{E}) and    Lemma \ref{DIFFERENTIAL} (ii), one obtains
$$(\Psi_*)_{H}V_\a+ [C,\,\Psi(H)] =  (\Psi_*)_{H}E_a^{-1}F_a V.$$
As a result,
$$-dd^c_J \rho_J (\widetilde{X}_z,\widetilde W_z)= %d\mu_J^X(\widetilde {iV}_z)=
 B \big (X\,,(\Psi_*)_{H} E_a^{-1}F_a{V} \big )\,, $$
 %$$ =- {\rm Im} B\big ( X,\,(\Psi_*)_{H} E_a^{-1}F_a \overline {iV} \big )\,,$$
 and, for $\,Z=X+iY$, one has   
$$\,-dd^c_J\rho_J(\widetilde Z_z,\,\widetilde W_z)= -dd^c_J\rho_J(\widetilde X_z,\,\widetilde W_z)+dd^c
_J\rho_J(\widetilde Y_z,\,\widetilde {iW_z})=$$
%-dd^c(\widetilde X_z,\,\widetilde W_z) -idd^c(\widetilde Y_z,\,\widetilde{W}_z)=
$$=- {\rm Im}B \big ( Z,\,(\Psi_*)_{H} E_a^{-1}F_a \overline W \big )\, .$$
From  Lemma \ref{FORMS2}(i), it follows    that
$\,-dd^c_J\rho_J =\omega_J,$
as desired.
\end{proof}

\bn

The plurisubharmonicity of $\,\rho_J\,$  will also be proved in Proposition
\ref{POSITIVE} where we show that  $\,\omega_J(\,\cdot\,, J\,\cdot\,)\,$ is positive definite.

%-------------------------------------------------------------------
%-----------------PROOF----------------------------
%-------------------------------------------------------------------

\bigskip
\section{Proof of the theorem}
\label{PROOF}
\bigskip

In this section we carry out the proof of the main theorem, mainly by collecting
 results proved in the previous sections. As a preliminary step 
 we show that the forms defined in Section \ref{STRUCTURE} are K\"ahler with respect to the corresponding $($almost$)$ complex structures.

%-----------------Riemannian metric----------------------------

\begin{prop}
\label{POSITIVE}
The $\,G$-invariant forms $\,\omega_I, \, \omega_J, \, \omega_K\,$ are K\"ahler with respect to the corresponding
$($almost$\,)$ complex structures $\,I,\,J,\,K$. Moreover, they define the same Riemannian metric
$$g(\,\cdot\,,\,\cdot\, )=\omega_I (\,\cdot\,,I\,\cdot\, )=\omega_J(\,\cdot\,,J\,\cdot\,)=
 \omega_K(\,\cdot\,,K\,\cdot\, )\,.$$ 
\end{prop}

 \begin{proof} 
 The invariance of $\,\omega_I$, $\,\omega_J$, $\,\omega_K\,$
 with respect to the corresponding almost complex structures was shown in
 Lemma \ref{PROPERTIFORMS}.
Their closeness was proved  in Lemma \ref{FORMS2}(iii) and
Proposition \ref{POTENTIAL}. From Definition \ref{FORMS} it is easy to check that 
 $\,\omega_I (\,\cdot\,,I\,\cdot\, )=\omega_J(\,\cdot\,,J\,\cdot\,)=
 \omega_K(\,\cdot\,,K\,\cdot\, )$.  In order to prove that the forms define a
 Riemannian metric, note that 
 $$\omega_J((ga)_*X,\,J(ga)_*iV)={\rm Im}B(I_0X,F_aI_0F_{a^{-1}} V)=0\,,$$
 for every $\,X \in \p\,$ and $\,iV \in i\p$. As  $\,\omega_J$ is $J$-invariant,
  it is enough to check  that 
$$\omega_J((ga)_*X,J(ga)_*X)={\rm Re}B(I_0X,F_aI_0F_{a^{-1}} X)>0$$
for every $\,X \in \p \setminus\{0\}$.
Since  the blocks of  decomposition  (\ref{DECO}) are eigenspaces of the map $\,F_a\,$ (see (\ref{EFFE})) and    are permuted by the complex structure $\,I_0\,$ (see (\ref{CPLX0})), it is sufficient to compute
 $\, {\rm Re}B(I_0P^\alpha,F_aI_0F_{a^{-1}} P^\alpha),$  for $\alpha  \in \Sigma^+\cup\{0\}$.
For this, note that if $\,I_0P^\alpha \in \p[\beta]\,$ (with the convention $\,\p[0]=\a\,$), one obtains  
$$ \textstyle {\rm Re}B(I_0P^\alpha,\,F_aI_0F_a^{-1}P^\alpha)=\frac{\cos \beta(H)}{\cos
\alpha(H)}B(I_0P^\alpha,\,I_0P^\alpha)=\frac{\cos \beta(H)}{\cos
\alpha(H)}B(P^\alpha,\,P^\alpha)\,,$$
which is strictly positive for all $\,H\in\Omega$.  \end{proof}

\bigskip

The next two lemmas are concerned with the uniqueness question for an arbitrary 
 $G$-invariant hyper-K\"ahler structure $({\mathcal I}, \,{\mathcal J}, \,{\mathcal K},\,
 \omega_{\mathcal I},\,\omega_{\mathcal J}, \,\omega_{\mathcal K})$ with the property that $\,{\mathcal J}=J_{ad}\,$ and the restriction of the K\"ahler structure $({\mathcal I},\,\omega_{\mathcal I})$ to $\,\p\cong T_{eK^\C}G/K\,$ coincides
 with the standard K\"ahler structure $(I_0,\,\omega_0)$ of $\,G/K$.

%------------------------forms \omega_I  and \omega_K ----------OLOSIMPLECTIC

\bigskip
\begin{lem}
\label{OLOSIMPLECTIC} 
Let $\,G/K\,$ be an irreducible non-compact  Hermitian symmetric space.
Assume that $\,\omega_{\mathcal I}\,$ and $\,\omega_{\mathcal K}\,$ 
are elements of a $\,G$-invariant, hyper-K\"ahler structure on $\,\Xi\,$
 such that $\,{\mathcal J}=J_{ad}$.
Then 
the $\,{\mathcal J}$-holomorphic symplectic form
$\,\omega_{\mathcal I}-i\omega_{\mathcal K}\,$ 
is locally
$\,G^\C$-invariant. 
\end{lem}

\medskip
\begin{proof} Recall that the local action of  $\,G^\C\,$ on 
$\,\Xi\,$ is ${\mathcal J}$-holomorphic and  the complex form $\,\omega_{\mathcal I}-i\omega_{\mathcal K}\,$
is $\,G$-invariant and  $\,{\mathcal J}$-holomorphic. Since $\,G\,$ is a real form of $\,G^\C$, the result follows from
the analytic continuation principle. 
\end{proof}

%----------------------END PROOF
\medskip

Denote by  $\, I_0: \p^\C \to \p^\C\,$ and by $\,\bar I_0: \p^\C \to \p^\C\,$ the 
linear and the anti-linear extension
of $\,I_0:\p \to \p$, respectively.
Then $\,\bar I_0 Z= I_0 \overline Z$.

%----------------------INITIAL

\bigskip
\begin{lem}
\label{INITIAL}
Assume that $\,\omega_{\mathcal I}$, $\,\omega_{\mathcal K}$ and $\,{\mathcal J}\,$ 
are elements of a $\,G$-invariant, hyper-K\"ahler structure on $\,\Xi\,$ with
$\,{\mathcal J}=J_{ad}\,$ and such that the K\"ahler structure $\,({\mathcal I},\,\omega_{\mathcal I})\,$ coincides with
$\,(I_0,\,\omega_0)\,$ when restricted to $\,G/K$. 
Then the restrictions of $\,{\mathcal I}\,$ and $\,\omega_{\mathcal I}\,$ to $\,\p^\C \cong T_{eK^\C}G^\C/K^\C\,$ coincide with  
$\,\bar I_0\,$ and $\,{\rm Re}B(I_0 \,\cdot\,, \, \cdot\,)$, respectively.
\end{lem}

\medskip
\begin{proof} Recall that the restriction of $\,{\mathcal J}=J_{ad}\,$
to $\,\p^\C\,$ is multiplication by $\,i$.
Since $\,{\mathcal I}{\mathcal J}=-{\mathcal J}{\mathcal I}$, for every $\,X \in \p\,$ one has $\,{\mathcal I}iX =-i{\mathcal I}X=-iI_0X$. This says
that the restriction of $\,{\mathcal I}\,$ to $\,\p^\C\,$ coincides with $\,\bar I_0$.

In a hyper-K\"ahler structure the form $\,\omega_{\mathcal I}$ is anti-${\mathcal J}$-invariant.
This determines its restriction  to $\,i\p={\mathcal J}\p$.

Then, in order to show that its restriction
 to $\,\p^\C\,$ coincides with $\,{\rm Re}B(I_0 \,\cdot\,, \, \cdot\,)$,
we are left to show that 
$\,\p\,$ and $\,{\mathcal J}\p\,$ are $\,\omega_{\mathcal I}$-orthogonal.

For this recall that, by Lemma \ref{OLOSIMPLECTIC},
for every $\,t \in \R\,$
the form $\,\omega_{\mathcal I}\,$ is $\,\exp itZ_0$-invariant.
Since for $\,X, \,Y \in \p\,$ one has 
$$\Ad_{\exp itZ_0}X=\cosh t\,X + \sinh t \,iI_0X \quad  {\rm and} \quad 
\Ad_{\exp itZ_0}Y=\cosh t\,Y + \sinh t \,iI_0Y\,,$$ it follows that 
$$\omega_{\mathcal I}(X,\,Y)=\omega_{\mathcal I}(\cosh tX + \sinh t iI_0X ,\,\cosh tY + \sinh t iI_0Y)=$$
$$ \cosh ^2 t \,\omega_{\mathcal I}(X ,\,Y) -\sinh^2 t \,\omega_{\mathcal I}(I_0X ,\,  I_0Y) +$$
$$\cosh t \sinh t \,\omega_{\mathcal I}(X ,\,{\mathcal J}I_0Y)+
\cosh t \sinh t \,\omega_{\mathcal I}({\mathcal J} I_0X ,\,Y )
\,.$$
Thus $\,\omega_{\mathcal I}(X ,\,{\mathcal J}I_0Y)+\omega_{\mathcal I}({\mathcal J} I_0X ,\,Y )=0\,$
which, by
$$\omega_{\mathcal I}(X ,\,{\mathcal J}I_0Y )=\omega_{\mathcal I}({\mathcal J} X ,\,I_0Y )=
\omega_{\mathcal I}({\mathcal J} I_0X ,\,Y )\,,$$
is equivalent to $\,\omega_{\mathcal I}({\mathcal J} X ,\,I_0Y )=0\,$ for every $\,X,\,Y\,$ in $\,\p$. That is,
$\,\p\,$ and $\,{\mathcal J}\p\,$ are $\,\omega_{\mathcal I}$-orthogonal, as wished.
\end{proof}

\bigskip
\begin{remark}
\label{INVARIANTIK}
Under the  assumptions of Lemma \ref{INITIAL}, 
the local $\,G^\C$-invariance of $\,\omega_{\mathcal I}\,$ and $\,\omega_{\mathcal K}\,$
$($Lemma \ref{OLOSIMPLECTIC}$\,)$ implies that  
$$\,\omega_{\mathcal I}((ga)_*Z , \,(ga)_*W )= {\rm Re} B(I_0Z ,\,  W )\,,$$
and 
$$\,\omega_{\mathcal K}((ga)_*Z , \,(ga)_*W )=\omega_{\mathcal K}(Z ,\,W ) =\omega_{\mathcal I}({\mathcal I}Z ,\,
{\mathcal I}{\mathcal J}W )=\omega_{\mathcal I}(Z ,\,{\mathcal J}W )
=-{\rm Im} B(I_0Z ,\,  W )\,.$$
\end{remark}

 %---------------------Theorem

\bigskip
\noindent
{\bf Proof of the main Theorem}. 
Let $\,J=J_{ad}\,$ be the adapted complex structure on $\,\Xi$,  let $\,I\,$ be  the almost complex structure defined in Section \ref{I} and   $\,K :=IJ$. Then the usual algebraic properties 
 $$I^2=J^2=K^2=-Id=IJK $$
follow directly  from Definition \ref{CPLX}.
By Proposition \ref{POSITIVE}, the 2-forms $\,\omega_I$, $\,\omega_J\,$ and $\,\omega_K\,
$ defined in Section \ref{STRUCTURE} are K\"ahler with respect to the corresponding almost complex structures and 
all define the same Riemannian metric. 
In addition, they are closed in view of Lemma \ref{FORMS2}(iii)  and Proposition  \ref{POTENTIAL}.
Now Lemma 6.8 in \cite{Hit87} implies that $\,I$, $\,J\,$ and $\,K\,$ are   integrable.
This concludes the proof of the existence of a hyper-K\"ahler structure on $\,\Xi\,$ with the required properties.
The proof of Part (c) in the main theorem can be found in Section \ref{OMEGAI}.

\sn
Finally, we outline a proof of  uniqueness of the adapted hyper-K\"ahler structure. 
Let
$$({\mathcal I},\,{\mathcal J},\,{\mathcal K},\,\omega_{\mathcal I},\,\omega_{\mathcal J},\, \omega_{\mathcal K})$$
 be an
arbitrary 
 $\,G$-invariant hyper-K\"ahler structure  with the property that $\,{\mathcal J}=J_{ad}\,$ and the restriction of the
 K\"ahler structure $\,({\mathcal I},\,\omega_{\mathcal I})\,$ to $\,\p\,$ coincides with the standard K\"ahler structure
 $\,(I_0,\,\omega_0)\,$ of $\,G/K$.
By Lemma \ref{OLOSIMPLECTIC} and Lemma \ref{INITIAL},  the restriction of $\,({\mathcal I},\,\omega_{\mathcal I})\,$
to $\,\p^\C\,$ necessarily coincides with 
$\,(\bar I_0, \,{\rm Re}B(I_0\,\cdot\,,\,\cdot\,))$.
Moreover, the forms $\,\omega_{\mathcal I}\,$ and $\,\omega_{\mathcal K}\,$, being locally $\,G^\C$-invariant, are uniquely determined everywhere on $\,\Xi\,$ (see Rem. \ref{INVARIANTIK}).
Then the relations 
 $\,{\mathcal K}={\mathcal I}{\mathcal J}\,$ and $\,\omega_{\mathcal J}(\cdot,\cdot )
 =\omega_{\mathcal I}({\mathcal I}\cdot ,\cdot)\,$   show that the hyper-K\"ahler structure is uniquely determined if the
  complex structure $\,{\mathcal I}\,$ is.
In turn, because of  its $\,G$-invariance, $\,{\mathcal I}\,$ is uniquely determined by the map
 $\,\overline L:\Omega\to GL_\R(\p^\C)$, given by $\,H\to \overline L_H$, which describes  $\,{\mathcal I}\,$ along the slice.
 That is, 
 $${\mathcal I}a_*Z=a_*\overline L_H  Z.$$
 In our hyper-K\"ahler setting, the operator $\,\overline L_H\,$ is an anti-linear anti-involution and 
 the form $\,{\rm Re}B(I_0 \, \cdot, \,\cdot, \,)\,$ is $\,\overline L_H$-invariant.
Moreover, the condition $\,d\omega_{\mathcal J}=0\,$ yields  a system of first order differential equations in the real analytic components of $\,\overline L_H\,$ with initial conditions $\,\overline L_{0}=\bar I_0$. Then the uniqueness
of the solution can be obtained by applying Cauchy-Kowaleskaya theorem. In the case of $\,G=SL_2(\R)\,$ the details of this strategy
are carried out in Appendix A.
\qed

%-------------------------------------------------------------------
%-----------------POTENTIAL omega_I----------------------------
%-------------------------------------------------------------------

\bigskip
\section{A potential for $\,\omega_I$.}
\label{OMEGAI}

\bigskip
In this section we determine a $\,G$-invariant function  $\,\rho_I\,$ with the property that 
$\,\omega _I= -dd^c_I \rho_I + p^* \omega_0$, where $\,\omega_0\,$ is the standard $\,G$-invariant  
K\"ahler form on $\,G/K$.  
Since the projection $\,p: \Xi \to G/K\,$ is holomorphic  (Prop. \ref{OLOMAP}),
a potential of $\,p^*\omega_0\,$ is given by the pull-back  of a potential 
$\,\rho_0\,$ of  $\,\omega_0$. Then $\,\rho_0 \circ p+\rho_I\,$ is a potential of $\,\omega_I$.
 %An
%$\,N$-invariant potential and a $\,K$-invariant potential of $\,\omega_0\,$ are exhibited in section \ref{OMEGA0}.

 We adopt the same strategy used in Section \ref{POTMOMENT}.
As a preliminary step, for a class of smooth $\,G$-invariant functions $\,\rho: \Xi \to \R$, we determine the associated function $\,\mu^X : \Xi \to \R\,$ 
defined by $\,\mu^X(z):=d_I^c \rho(\widetilde X_z)$, 
 for $\,X \in \g$ (cf. Sect. \ref{POTMOMENT}).

 %-------------------MOMENTST------------------
 
 \bigskip
\begin{lem}
\label{MOMENTII} Given a smooth function $\,f:\R \to \R$, consider the 
$\,G$-invariant function $\,\rho:\Xi \to \R\,$ defined  by
$$\,\rho(gaK^\C):=  -\textstyle \frac{1}{4}\sum_{j=1}^r f(\lambda_j(H))B(A_j,A_j)\,.$$
For  $\,X\,$ in $\,\g\,$ one has 
$$\,\mu^X(gaK^\C) =\textstyle {\frac{1}{2}\sum_{j=1}^r} \widetilde f (\lambda_j(H))
B (\Ad_{g^{-1}} X ,\,[I_0A_j,\,H])\,,$$
where  $\,\widetilde f(t) = \frac{\sin t}{t\cos t}f'(t)$.
\end{lem}

\medskip
\begin{proof}  The $\,G$-invariance of $\,\rho\,$ and $\,I$ implies that, 
$$ \, d_I^c\rho(\widetilde X_{g\cdot z})=d_I^c\rho( \widetilde {\Ad_{g^{-1}}X}_{z})\,$$
for every $\,z \in \Xi\,$ and $\,g \in G$. So it is sufficient to prove the lemma for  $g=e$ and $z=aK^\C$, with $a= \exp iH$ and $H \in \Omega'$ (see (\ref{REGULAR})).

Take $\,X \in \p$. Since
$\,I \widetilde X_z= \widetilde {I_0X}_z$,
the $\,G$-invariance of $\,\rho\,$ implies  that 
$\,d_I^c \rho (\widetilde X_z)=d \rho (\widetilde {I_0X}_z)=0.$

Next, take $X \in \k$ and decompose it as $X=M +\sum_\alpha K^\alpha$, as
in (\ref{DECO}), with $M \in \m$ and $K^\alpha= X^\alpha+ \theta X^\alpha\,$ in $\,\k[\alpha]$.
Since $\widetilde M_z=0$ for $z=aK^\C$, one has $\,d_I^c \rho (\widetilde M_z)=0$.

For $\,\alpha \not= \lambda_1, \dots, \lambda_r$ and 
$\,K^\alpha=X^\alpha+ \theta X^\alpha\,$ in $\,\k[\alpha]$, set
$\,P^\alpha= X^\alpha- \theta X^\alpha\,$ in $\,\p[\alpha]$. Also set 
$\,I_0P^\alpha=:P^\beta=X^\beta- \theta X^\beta$ in $\p[\beta]$
and $K^\beta=X^\beta+\theta X^\beta$. Note that $\,\beta \not=0$.
Then by (\ref{TILDEVSHAT}),  Lemma \ref{CURVESSTA}(iii) and (\ref{CPLXBIS}) one has 
$$d_I^c \rho( \widetilde {K^\alpha}_z )=
-d \rho (Ia_* \sin \alpha(H) iP^\alpha )=d \rho(a_*\sin \alpha(H) i
F_aI_0F^{-1}_aP^\alpha  )=$$
$$
d \rho \big ( a_* \textstyle  \frac {\sin \alpha(H)\cos \beta (H)}{
\cos \alpha(H)}iP^\beta \big )=
-d \rho \big ( \textstyle \frac {\sin \alpha(H)\cos \beta (H)}{ \sin \beta(H)
\cos \alpha(H)} \widetilde{K^\beta}_z \big )= 0\,.$$

\smallskip
For $\alpha=\lambda_l$, with $\,l=1, \dots\, r$, one has
$$d_I^c \rho  ( \widetilde{ K^l}_z)=
d \rho \big (\textstyle  a_* \sin \lambda_l(H) iF_a
I_0F^{-1}_aP_l \big ) =d \rho \big(
a_*\frac{\sin \lambda_l(H)}{\cos \lambda_l(H)}iA_l \big)=$$
$$
\textstyle
-\frac{1}{4} \dds  \sum_j f \big(\lambda_j \big ( H + s \frac{\sin \lambda_l(H)}{\cos \lambda_l(H)} A_l\big)\big )B(A_j,A_j)=$$
$$\textstyle -\frac{1}{2}\frac{\sin \lambda_l(H)}{\cos \lambda_l(H)}f'(\lambda_l(H))B(A_l,\,A_l)\,.$$
Since
\begin{equation}
\label{PASSAGGI}
\begin{aligned}
\textstyle
B (A_l ,\,A_l )=&B (I_0 A_l ,\,I_0 A_l)= \cr
\textstyle
-\frac{1}{\lambda_l(H)}B  ([H,\,K^l] ,\,\,I_0 A_l )=&
\textstyle -\frac{1}{\lambda_l(H)}B (K^l ,\,[I_0A_l,\, H])\,,
\end{aligned}
\end{equation}
the above computation yields  $$\,\textstyle d_I^c\rho(\widetilde X_{aK^\C}) = \frac{1}{2}\sum_j
\widetilde f(\lambda_j(H))B (X ,\,[I_0A_j,\, H])\,,$$
  as desired.
  \end{proof}

%---------------------POTENTIALII
\bigskip
\begin{prop}
\label{POTENTIALII} Fix a function $\,f_I: \R \to \R\,$ with the property that
$\, \frac{\sin t}{\cos t} f_I'(t)= \cos t -1$.
The $\,G$-invariant function $\,\rho_I:\Xi \to \R\,$ defined by 
$$\, \textstyle \rho_I(ga K^\C):=  -\frac{1}{4}\sum_{j=1}^r f_I(\lambda_j(H))B(A_j,A_j)\,,$$
is a potential for   $\,\omega_I- p^*\omega_0$.
In particular the form $\,\omega_I=-dd_I^c \rho_I+p^*\omega_0\,$ is closed
$(cf.\  Prop.\ \ref{FORMS2}\,(iii) )$.
\end{prop}

\medskip
\begin{proof} 
We are going to show that $-dd_I^c \rho_I= \omega_I-p^*\omega_0$. By the $G$-invariance 
of the forms involved it is sufficient to prove the statement for
 $z=aK^\C$, with $a= \exp iH$ and $H \in \Omega'$ (see (\ref{REGULAR})).
 Observe that  \begin{equation}\label{OMEGA00}
 p^*\omega_0(\widetilde Z_z, \, \widetilde W_z)= \omega_0(X, \,U)=B(I_0X, \,U)\,,
 \end{equation}
 for every $Z=X+iY$ and $W=U+iV$ in $\p^\C$.
 Also recall that by Lemma \ref{DIFFER}, for every $X \in \p$ and $W\in \p^\C$
  one has
$$-dd_I^c \rho_I(\widetilde X_z, \, \widetilde W_z)=d\mu_I^X(\widetilde W_z)\,,$$
where the map $\mu_I^X$ can be computed as in Lemma \ref{MOMENTII}.

Fix  a point $z=aK^\C$ in the slice in $\Xi'$. We perform the computation exploiting the block
decomposition of $\p$ given in (\ref{DECO}) with the convention $\p[0]=\a$.
   
 \nmedskip
 {\boldmath \bf$\bullet \ \ -dd_I^c \rho_I= \omega_I-\pi^*\omega_0\ $ on
$\ {a_*\p \times a_*\p}$.}

\bigskip
\noindent
\nmedskip
For  $\,X,\,U \in \p$, by Lemma \ref{DIFFER} and Lemma \ref{MOMENTII} one has 
 $$
 \textstyle 
 -dd_I^c \rho_I(\widetilde X_z, \, \widetilde U_z)=
 \dds \mu_I^{X}( \exp sU aK^\C)= $$
 $$= \textstyle \frac{1}{2}
  \dds \sum_j\widetilde f_I (\lambda_j(H))
B (X-s[U,\,X] ,\,[I_0A_j,\,H])=$$
$$ \textstyle- \frac{1}{2}\sum_{j=1}^r\widetilde f_I (\lambda_j(H))
B  (X ,\,[[I_0A_j,\,H], U] )\,.$$
Moreover
$$B  (X ,\,[[I_0A_j,\,H], U] )=-
B  (X ,\,[[U, \,I_0A_j],\,H] +[[H,\, U],\,I_0A_j] )=$$
$$-B([H,\,X] ,\,[U,\,I_0A_j])-
B ([I_0A_j,\,X],\, [H, U] )=$$
$$
-B([H,\,X] ,\,[A_j,\,I_0U])+
B ([A_j,\,I_0X],\, [H, U] ) =$$
$$B ([A_j,\,[H,\,X]],\,I_0U)-
B  (I_0X,\, [A_j,\,[H,\,U]] )\,$$
which, for $X=P^\alpha \in \p[\alpha]$ and $U=Q^\beta \in \p[\beta]$, becomes 
$$ -\big ( \alpha(A_j)\alpha(H)+\beta(A_j)\beta(H) \big )
B  (I_0P^\alpha ,\,Q^\beta)\,.$$
Thus one obtains 
$$\textstyle
-dd_I^c \rho_I(\widetilde {P^\alpha}_z, \,\widetilde {Q^\beta}_z)=\frac{1}{2}\sum_j\widetilde f_I (\lambda_j(H)) 
\big ( \alpha(A_j)\alpha(H)+\beta(A_j)\beta(H) \big )
B  (I_0P^\alpha ,\,Q^\beta)\,.$$
It is clear that  
$-dd_I^c \rho_I(\widetilde {A}_z, \,\widetilde {A'}_z)
=0\,,$ for all $A, \,A'$ in $\a$.
 In view of relations (\ref{DECO}), (\ref{CPLX0}) and (\ref{CPLXBIS}),
 we are left to check the following  cases.

\bigskip
\noindent
 {\bf Case \boldmath $\,\alpha= 0\,$ and  $\,  \beta=\lambda_k$.}
 The above computation and our assumption on $\,f_I\,$ imply
  $$\textstyle -dd_I^c \rho_I(\widetilde {A}_z, \,\widetilde {P^k}_z)=
  \frac{1}{2}\widetilde f_I (\lambda_k(H)) 
2 \lambda_k(H) B (I_0A ,\,P^k )=
(\cos  \lambda_k(H)-1)B (I_0A ,\,P^k)=$$
$$=\omega_I(\widetilde {A}_z, \,\widetilde {P^k}_z)-p^*\omega_0(\widetilde {A}_z, \,\widetilde {P^k}_z)\,,
$$
where the last identity follows from  (\ref{OMEGA00}) and 
$$\cos \lambda_k(H)B  (I_0A ,\,P^k)=
B (I_0F_aA ,\,F_aP^k)=\omega_I(\widetilde {A}_z, \,\widetilde {P^k}_z)\,.$$

\medskip
\noindent
 {\bf \boldmath Case  $\, \alpha=e_k+e_l\,$ and  $\,  \beta= e_k-e_l$.} Since 
 $\,\lambda_k= \alpha+\beta\,$ and $\,\lambda_l= \alpha-\beta$,
 one has  
 $$
 \textstyle -dd_I^c \rho_I(\widetilde {P^\alpha}_z,
 \,\widetilde {Q^\beta}_z)=
  \frac{1}{2}\big(\widetilde f_I (\lambda_k(H)) (\alpha(H) + \beta(H))+\widetilde f_I (\lambda_l(H)) 
  (\alpha(H) - \beta(H))
  \big )  B \big (I_0P^\alpha ,\,Q^\beta\big)=$$
$$ \textstyle \frac{1}{2}\big(\widetilde f_I (\lambda_k(H)) \lambda_k(H)+\widetilde f_I (\lambda_l(H)) 
  \lambda_l(H)
  \big )  B  (I_0P^\alpha ,\,Q^\beta)=$$
  $$ \textstyle
\big (\frac{\cos  \lambda_k(H)+\cos  \lambda_l(H)}{2}-1\big )B (I_0P^\alpha ,\,Q^\beta)\,.$$
By the trigonometric identity
$$\cos  2t_k+\cos  2t_l= 2\cos (t_k +t_l) \cos (t_k-t_l)\,$$
one obtains
$$\textstyle \frac{1}{2}(\cos  \lambda_k(H)+\cos  \lambda_l(H))B (I_0P^\alpha ,\,Q^\beta)=
B  (I_0F_aP^\alpha ,\,F_aQ^\beta)=\omega_I(\widetilde {P^\alpha}_z, \,\widetilde {Q^\beta}_z)\,.$$
Hence 
$$ \textstyle -dd_I^c \rho_I(\widetilde {P^\alpha}_z,
 \,\widetilde {Q^\beta}_z)=\omega_I(\widetilde {P^\alpha}_z, \,\widetilde {Q^\beta}_z) - p^*\omega_0(\widetilde {P^\alpha}_z, \,\widetilde {Q^\beta}_z)\,.$$

\medskip
\noindent
 {\bf \boldmath Case $\, \alpha=e_k=\beta$.}
Since $\lambda_k= 2 \alpha$, one  has 
$$
\textstyle -dd_I^c \rho_I(\widetilde {P^\alpha}_z, \,\widetilde {Q^\alpha}_z)=
  \frac{1}{2}
  \widetilde f_I(\lambda_k(H) )\lambda_k(H) 
   B(I_0P^\alpha ,\,Q^\alpha)=
\frac{1}{2} \big(\cos  \lambda_k(H)-1\big )B  (I_0P^\alpha ,\,Q^\alpha)=$$
$$(\cos^2  \alpha(H)-1)B (I_0P^\alpha ,\,Q^\alpha)=
\omega_I(\widetilde {P^\alpha}_z, \,\widetilde {Q^\alpha}_z) - p^*\omega_0(\widetilde {P^\alpha}_z, \,\widetilde {Q^\alpha}_z)
$$
where in the last equality one uses 
$$\cos^2  \alpha(H)B (I_0P^\alpha ,\,Q^\alpha)=
B (I_0F_aP^\alpha ,\,F_aQ^\alpha)=\omega_I(\widetilde {P^\alpha}_z, \,\widetilde {Q^\alpha}_z)\,.$$

%-----------------------------\p contro i\p
 \nmedskip
 { \boldmath \bf$\bullet \ \ -dd_I^c \rho_I= 0\ $ on
$\ {a_*\p \times a_*i\p}$.}

\nbigskip

Take $\,X\in\p\,$ and $\,iV \in i\p$. By Lemma   \ref{CURVESSTA}(i)
one has $\,\widetilde{iV}_z=
\dds \exp sC\exp i(H+sV_\a)K^\C$, with $\,C \in \k$. Thus
$$-dd^c\rho_I(\widetilde X_z,\,\widetilde{iV}_z)=
\textstyle
\dds \mu_I^X(\exp sC\exp i(H+sV_\a)K^\C) =$$
%$$= \dds \mu^{Ad_{\exp -sC}X}(\exp i(H+sX_\a)K^\C)$$
$$
\textstyle{1\over 2}
\dds \sum_j\widetilde f(\lambda_j(H+sV_\a))B(Ad_{\exp -sC}X,[I_0A_j,\,H+sV_\a])=$$
$$
\textstyle
{1\over 2}\sum_j \widetilde f'(\lambda_j(H)) \lambda_j(V_\a)B(X,\,[I_0A_j,\,H])
+\widetilde f(\lambda_j(H))
\big ( B([X,C],[I_0A_j,\,H])+ B(X,[I_0A_j,\,V_\a])\big)\,.$$
All the above summands are zero, since $\,[\k,\p]\subset \p$, $\,[\p,\p]\subset \k\,$ and $\,B(\p,\k)\equiv 0$.

\bn

 \nmedskip
{\boldmath \bf
 $\bullet \ \ -dd_I^c \rho_I= \omega_I\ $
 on
$\ {a_*i\p \times a_*i\p}$.}

 \bigskip
\noindent
 {\boldmath  \bf Case $\,\alpha=0=\beta$.}
 We show that 
$\,-dd_I^c \rho_I(\widetilde {iA}_z, \,\widetilde {iA'}_z)=0\,$ for every $\,A,\,A' \in \a$.
Since $\,[\widetilde {iA}, \,\widetilde {iA'}]=\widetilde{[{iA}, \,{iA'}]}=0$,  the usual formula for the exterior derivation
applied to the 1-form $\,d_I^c \rho_I\,$ yields
$$dd_I^c \rho_I(\widetilde {iA}_z, \,\widetilde {iA'}_z)=\widetilde {iA}_z(d_I^c \rho_I(\widetilde {iA'} ))-
\widetilde {iA'}_z(d_I^c \rho_I(\widetilde {iA} ))=$$
$$ \textstyle
=\ddt d\rho_I\big (I\dds \exp siA' \exp tiA aK^\C-I\dds \exp siA \exp tiA' aK^\C \big )=$$
$$=\textstyle
\ddt d\rho_I\big (\dds \exp siI_0A' a\exp tiA K^\C -\dds  \exp siI_0A a\exp tiA' K^\C \big )\,.$$
By Lemma \ref {CURVESSTA}(ii), for $\,A=A_k\,$ and $\,A'=A_l\,$ the above expression 
equals to 
$$
-\textstyle
\ddt \dds \rho_I \big ( \exp s  \frac {\cos \lambda_l(H+tA_k)}{\sin \lambda_l
(H+tA_k)} K^l \exp i(H+tA_k) K^\C \big ) $$
$$+ \textstyle \ddt
\dds \rho_I \big ( \exp s  \frac {\cos \lambda_k(H+tA_l)}
{\sin \lambda_k
(H+tA_l)} K^k  \exp  i(H+tA_l) K^\C \big )\,,$$
which vanishes by the $\,G$-invariance of $\,\rho_I$.

\bigskip
\noindent
 {\bf Case  \boldmath $ \, \alpha \not= 0\,$ and  $\,   \beta=0$.}
Take $\,iP^\alpha \in i\p$, with $\,\alpha \not=0$, and $A_k \in \a$. Then,
by Lemma \ref{CURVESSTA}(iii)
$$\textstyle 
-dd_I^c\rho_I(a_*iP^\alpha,\,a_*iA_k)=
{1\over \sin\alpha(H)}dd_I^c\rho_I((\widetilde{K^\alpha})_z,(\widetilde{A_k})_z)=$$
 $$
 \textstyle 
 -{1\over \sin\alpha(H)} \dds \mu_I^{K^\alpha}(\exp isA_k \, aK^\C)=$$
$$\textstyle 
-{1\over 2\sin\alpha(H)} \dds \sum_j
\widetilde f(\lambda_j(H+sA_k))B(K^\alpha,[I_0A_j,\,H+sA_k])=$$
$$\textstyle 
-{1\over 2\sin\alpha(H)}\sum_j\widetilde f'(\lambda_j(H))\lambda_j(A_k)B(K^\alpha,[I_0A_j,\,H])+\widetilde f(\lambda_j(H))B(K^\alpha,[I_0A_j,\,A_k])\,.
$$
Since $\lambda_j(A_k)=0=B(K^\alpha,[I_0A_j,\,A_k])$ for $j\not=k$, 
the above sum reduces to
$$
\textstyle 
-{1\over 2\sin\lambda_k(H)} \big ( \widetilde f'(\lambda_k(H))2 B(K^\alpha,\,[I_0A_k\,H])+
\widetilde f(\lambda_k(H))\big )B(K^\alpha,\,[I_0A_k\,A_k])=$$
$$=
\textstyle 
{1\over 2\sin\lambda_k(H)} \big ( \widetilde f'(\lambda_k(H))2\alpha(H) +\widetilde f(\lambda_k(H))
\alpha(A_k) \big )B(I_0P^\alpha,\,\,A_k).$$
This expression vanishes for $\alpha \not= \lambda_k$, while for $P^\alpha = P^k \in \p[\lambda_k]$
becomes
$$
\textstyle 
{1\over \sin\lambda_k(H)} \big ( \widetilde f'(\lambda_k(H))\lambda_k(H) +\widetilde f(\lambda_k(H))
 \big )B(I_0P^k,\,\,A_k).$$
Under our assumption $\,\widetilde f(t)= \frac {1}{t}(\cos t -1),$ one has
$\,\widetilde f'(t)= -\frac {1}{t^2}(\cos t -1) -\frac {\sin t}{t}$, and consequently
$$\textstyle 
{1\over \sin t}\big ( \widetilde f'(t)t +\widetilde f(t) \big )=-1\,.$$
Thus
$$
\textstyle 
-dd_I^c\rho_I(a_*iP^\alpha,\,a_*iA_k)=
-B(I_0P^k,\,A_k)=\omega_I(a_*iP^k,\,a_*iA_k)\,.$$

\bigskip
\noindent
 {\bf \boldmath Case  $\, \alpha\not=0\not=\beta$.}
 Next, consider  $\,iP^\alpha,\,iQ^\beta \in i\p$, with $\,\alpha \not=0\not=\beta$.
By Lemma \ref{CURVESSTA}(i) one has
$\, a_*iP^\alpha=-\textstyle \frac{1}{\sin \alpha(H)} \widetilde K^\alpha\,$ and
$\, a_*iQ^\beta=-\textstyle \frac{1}{\sin \beta(H)} \widetilde C^\beta\,,$
for appropriate $K^\alpha \in \k[\alpha]$  and $C^\beta \in \k[\beta]$. 
Then
$$-dd_I^c\rho_I(a_*iP^\alpha,\,a_*iQ^\beta)=\textstyle
\frac{1}{\sin\alpha(H) \sin \beta(H)}
\dds \mu_I^{K^\alpha}( \exp sC^\beta \, aK^\C)=$$
$$ \textstyle
-\frac{1}{2\sin\alpha(H) \sin \beta(H)}\sum_j \widetilde f (\lambda_j(H))
B \big ([C^\beta,\,K^\alpha] ,\,[I_0A_j,\,H] \big)\,=$$
$$ \textstyle
=-\frac{1}{2\sin\alpha(H) \sin \beta(H)}\sum_j \widetilde f (\lambda_j(H))
B \big (K^\alpha ,\,[[I_0A_j,\,H],\,C^\beta] \big)\,.$$

By writing
$$
B \big (K^\alpha ,\,[[I_0A_j,\,H],\,C^\beta] \big)=-B\big (K^\alpha ,\,[[C^\beta,\,I_0A_j],\,H]+[[H,\,C^\beta],\,I_0A_j]  \big)=$$
$$-\big (\alpha(H)\beta(A_j)+
\beta(H) \alpha(A_j)\big) B\big (I_0P^\alpha,\, Q^\beta  \big)\,$$
one  obtains 
$$-dd_I^c\rho_I(a_*iP^\alpha,\,a_*iQ^\beta)=$$
$$\textstyle \frac{1}{2 \sin \alpha(H)\sin\beta(H)}\sum_j\widetilde f(\lambda_j(H))\big(
\alpha(H)\beta(A_j)+\beta(H)\alpha(A_j)\big) B(I_0P^\alpha,\,Q^\beta).$$
In view of relations (\ref{DECO}), (\ref{CPLX0}) and (\ref{CPLXBIS}),
 we are left to check the following  cases.

\bigskip
\noindent
 {\bf \boldmath Case $\, \alpha=e_k+e_l\ $ and  $\  \beta= e_k-e_l$.} Since
  $\, \alpha +\beta=\lambda_k\,$ and  $\,-\alpha+\beta=-\lambda_l$, one has
  $$-dd_I^c\rho_I(a_*iP^\alpha,\,a_*iQ^\beta)=
  \textstyle
  {1\over {2\sin\alpha(H) \sin \beta(H)}} \big (\widetilde f(\lambda_k(H))\lambda_k(H)-\widetilde f(\lambda_l(H))\lambda_l(H)
\big )
B(I_0P^\alpha,\,Q^\beta)$$
$$ 
\textstyle
\frac{ \cos\lambda_k(H)-\cos\lambda_l(H)}{2\sin\alpha(H) \sin \beta(H)}B(I_0P^\beta,\, Q^\beta)=-B(I_0P^\beta,\,Q^\beta) =\omega_I(a_* iP^\alpha,\,a_*iQ^\beta),$$
due to the trigonometric  identity
 $\,\cos 2t_k-\cos 2t_l=-2\sin(t_k+t_l)\sin(t_k-t_l)$. 

\medskip
\noindent
 {\bf \boldmath Case  $\, \alpha=e_k=\beta$.}
One has
$$-dd_I^c\rho_I(a_*iP^\alpha,\,a_*iQ^\alpha)=
\textstyle
{1\over {2\sin^2\alpha(H)}} \widetilde f(\lambda_k(H))\lambda_k(H)B(I_0P^\alpha,\,Q^\alpha)$$
 $$\textstyle= {1\over { 2\sin^2\alpha(H)}}(\cos2\alpha(H)-1)B(I_0P^\alpha,Q^\alpha)$$
 $$=-B(I_0P^\alpha,\, Q^\alpha)=\omega_I(a_* iP^\alpha,\,a_*iQ^\alpha)\,.$$
\end{proof}

%-------------------------------------------------------------------
%-----------------ENDOFPROOF----------------------------
%-------------------------------------------------------------------

\smallskip
\section{Appendix A: a proof of uniqueness of the adapted hyper-K\"ahler structure for $\,G=SL_2(\R)$.}
\label{UNIQUESL2}

\bigskip
Here we carry out a proof of uniqueness of the adapted hyper-K\"ahler
structure in the case of $\,G=SL_2(\R)$, as announced in 
Section \ref{PROOF}.

Consider the map
$\,\p \times \k \times \Omega^+ 
\to \Xi$, given by $\,(U,C,H) \to \exp U \exp C \exp iH K^\C$, 
which is an analytic diffeomorphism of a neighborhood
of $\,\{0\}\times \{0\}\times \Omega^+\,$  onto its image $\,\Omega''\,$
(cf. \cite{KrSt05}, Cor. 4.2\,). On $\,\Omega''\,$ we consider the vector fields
$$\textstyle \widecheck A_{ \exp U \exp C \exp iHK^\C} :=
{d\over ds}\Big |_{s=0} \exp (U+sA) \exp C \exp iH K^\C \,,$$
$$\textstyle \widecheck {P}_{\exp U \exp C \exp iHK^\C}: = {d\over ds}\Big |_{s=0} \exp (U+sP) \exp C \exp iH K^\C\,,$$
$$\textstyle \widecheck {K}_{ \exp U \exp C \exp iHK^\C}: = {d\over ds}\Big |_{s=0} \exp U \exp (C+sK)
\exp iH
 K^\C \ ,$$
$$\textstyle \widecheck {iA}_{ \exp U \exp C \exp iHK^\C}: = {d\over ds}\Big |_{s=0} \exp (U) \exp C \exp i(H+sA) K^\C\,,$$

\sn
where  $\,A:=[\theta  E,\, E]$, $\,P=E - \theta E\,$ and $\,K=E + \theta E\,$.
In particular  $\,I_0A=-P$, $\,I_0P=A\,$. Moreover
$\,[A,\,K]=\alpha(A)P=2P\,$ and $\,[A,\,P]=\alpha(A)K=2K\,$   (see (\ref{NORMALIZ1}), (\ref{KJPJ}) and (\ref{CPLX0})). 
All above  vector fields commute, since they are push-forward of coordinate vector fields in the product
$\,\p \times \k \times \Omega^+ $.

\bigskip
\noindent
{\bf Proof of  uniqueness of the adapted hyper-K\"ahler structure} ($\,$case of $\,G=SL_2(\R)\,$). 
Let
$$({\mathcal I},\,{\mathcal J},\,{\mathcal K},\,\omega_{\mathcal I},\,\omega_{\mathcal J},\, \omega_{\mathcal K})$$
 be an
arbitrary 
 $\,G$-invariant hyper-K\"ahler structure  with the property that $\,{\mathcal J}=J_{ad}\,$ and the restriction of the K\"ahler structure
 $\,({\mathcal I},\,\omega_{\mathcal I})\,$ to $\,\p\,$ coincides with the standard K\"ahler structure $\,(I_0,\,\omega_0)\,$ of $\,G/K$. 
Consider the map  $\,\overline L:\Omega \to GL_\R(\p^\C)\, $ which describes  $\,{\mathcal I}\,$
 along the slice by $${\mathcal I}a_*Z=a_*\overline L_H  Z\,,$$ 
 where $\,H\in\Omega\,$ and $\,a= \exp iH$. 
 As observed in the proof of the main Theorem in Section \ref{PROOF}, we need to show that  
for every $\,H\,$ in $\,\Omega\,$ and $\,Z \in \p^\C$ one has $\,\overline L_H  Z=F_aI_0F_a^{-1}\overline Z$.
Note that 
\begin{equation}
\label{OMEGAL}
\omega_{\mathcal J}(a_*Z \,, \, a_*W)=\omega_{\mathcal I}({\mathcal J}a_*Z \,, \, {\mathcal I}a_*W)\,
 =-{\rm Im}B(I_0Z,\,\overline L_HW )\,.
\end{equation}

\nmedskip
{\it Claim.}  With respect to the  basis 
$$\{A,\,P,\,iA,\,iP \}$$
of $\,\p^\C$,
the anti-linear anti-involution $\,\overline L_H\,$ of $\,\p\,$ is represented by the matrix
$$
 \begin{pmatrix}
 a_1&\ a_2&b_1&\ 0 \cr
a_3&-a_1&0&\ b_1\cr
b_1&0&-a_1&-a_2\cr
0&b_1&-a_3&a_1
 \end{pmatrix}\,,$$
where $\,a_1$, $\,a_2$, $\,a_3\,$ and $\,b_1\,$ are real-analytic functions of 
 $\,H\,$ and $\, b_1^2+a_1^2+a_2a_3 =-1$.

\nmedskip
{\it Proof of the claim.}
Let
$$
 \begin{pmatrix}
 A&B \cr
C& D \end{pmatrix}\,$$

\nmedskip
be the representative matrix of $\,L_H\,$ with respect to
the above basis, which is compatible with  the decomposition
$\,\p \oplus i\p$. Since $\,{\mathcal I}{\mathcal J}=-{\mathcal J}{\mathcal I}\,$, it follows that
$\, \overline L_H {\mathcal J}Z=-{\mathcal J}\overline L_HZ\,$ for
every $\,Z\,$ in $\,\p$, i.e. $\,\overline L_H\,$ is anti-linear. 
This implies that $\,C=B\,$ and $\,D=-A$, where 
$$A=  \begin{pmatrix}
a_1 & a_2\cr
a_3 & a_4 \end{pmatrix}\, \qquad {\rm and}\qquad B= \begin{pmatrix}  b_1 & b_2 \cr
b_3& b_4 \end{pmatrix}\,.$$
Since 
 $\,\omega_{\mathcal J}(\, \cdot \, , \,\, \cdot \, )$ is skew-symmetric, (\ref{OMEGAL}) implies that 
$${\rm Im}B(I_0Z,\,\overline L_HW )=
-{\rm Im}B(I_0W,\,\overline L_HZ )=-{\rm Im}B(\overline L_HZ,\,I_0W )$$
for every $Z,\,W \in \p^\C$. 
As  $\ {}^t\! I_0 = -I_0$, one obtains
$$ -\begin{pmatrix}
 I_0& 0\cr
0& I_0
\end{pmatrix}
  \begin{pmatrix}
 0&Id \cr
Id& \ 0
\end{pmatrix}
 \begin{pmatrix}
 A&B \cr
B& -A \end{pmatrix}=-
 \begin{pmatrix}
{}^t\! A& {}^t B \cr
{}^t\! B& {-}^t\! A
\end{pmatrix}
\begin{pmatrix}
 0&Id \cr
Id& \ 0
\end{pmatrix}
\begin{pmatrix}
 I_0&0\cr
0& I_0
\end{pmatrix}
\,,$$
which implies 
$\,I_0A= {-}^t AI_0\,$ and $\,I_0B= {}^t BI_0$.
Thus the matrix realization of $\,\overline L_H\,$ is as  claimed and the relation 
$\, b_1^2+a_1^2+a_2a_3 =-1\,$ follows from the fact that 
$\,(\overline L_H)^2=-Id$. This concludes the proof of the claim.

\medskip
Then  in order to conclude the proof,
we need to show that the functions $\,a_1\,$ and $\,b_1\,$ identically vanish and $\,a_3(H)= -\cos \alpha(H)$\,
(recall that $I_0A=-P$ and $I_0P=A$). This will be done by showing that such functions
are solutions of a system of differential equations with initial conditions
$a_1(0)=0=b_1(0)$, $\,a_3(0)=-1$.  Without loss of generality, in the sequel
 we assume that the Killing form $\,B\,$ is normalized by  $\,B(A,A)=B(P,P)=1$.

\nsmallskip
\noindent
 {\bf \boldmath  $\bullet \ \,b_1\equiv 0$.}
Let $\,z=aK^\C \in \Xi''$, with $\,a=\exp iH$. 
Since the vector fields
$\,\widecheck A$, $\,\widecheck P$, $\,\widecheck {iA}\,$ commute
 and $\,\omega_{\mathcal J}\,$ is closed, 
the classical Cartan's formula gives
$$
d\omega_{\mathcal J}(\widecheck A_z,\,\widecheck P_z,\,\widecheck {iA}_z)=
\widecheck A_z\omega_{\mathcal J}(\widecheck P,\,\widecheck {iA})- \widecheck P_z\omega_{\mathcal J}(\widecheck A,\,\widecheck {iA} )+
\widecheck {iA}_z\omega_{\mathcal J}(\widecheck A,\,\widecheck P) 
=0\,.
$$
One has
$$ \textstyle
d\omega_{\mathcal J}(\widecheck A_z,\,\widecheck P_z,\,\widecheck {iA}_z)
=\ddt \omega_{\mathcal J}(\widecheck P_{\exp tAaK^\C},\,\widecheck {iA}_{\exp tAaK^\C})+$$
$$\textstyle
- \ddt \omega_{\mathcal J}(\widecheck A_{\exp tPaK^\C},\,\widecheck {iA}_{\exp tPaK^\C} )+
\ddt \omega_{\mathcal J}(\widecheck A_{\exp i(H+tA)K^\C} ,\,\widecheck P_{\exp i(H+tA)K^\C}) =
$$
\medskip
$$= \textstyle \ddt \omega_{\mathcal J}(\dds \exp (tA+sP)aK^\C,\,\dds \exp tA\exp i(H+sA)K^\C)+$$
$$\textstyle
-\ddt \omega_{\mathcal J}(\dds \exp (tP+sA)aK^\C,\,(\exp tPa)_*iA )+$$
$$+ \textstyle 
\ddt \omega_{\mathcal J}(\exp i(H+tA)_*A ,\,\dds \exp s P \exp i(H+tA)K^\C) =
$$
\medskip
$$= \textstyle \ddt \omega_{\mathcal J}(\dds {\exp tA\exp s(P-\frac{t}{2}[A,P] + O(t^2))aK^\C},\,(\exp tAa)_*iA )+$$
$$\textstyle
-\ddt \omega_{\mathcal J}(\dds \exp tP\exp s(A-\frac{t}{2}[P,A] + O(t^2))aK^\C,\,(\exp tPa)_*iA )+$$
$$+ \textstyle 
\ddt \omega_{\mathcal J}(\exp i(H+tA)_*A ,\,\exp i(H+tA)_*\cos \alpha(H+tA)P) 
$$
which, by (\ref{OMEGAL}), gives
$$ \textstyle \ddt \omega_{\mathcal J}(\dds \exp s(-t\alpha(A)K)aK^\C,\,a_*iA )+$$

$$-\textstyle 
\ddt \cos \alpha(H+tA){\rm Im}B(I_0A ,\,\overline L_{H+tA}P)=
$$
\medskip
$$= \textstyle  \omega_{\mathcal J}(a_* \alpha(A) \sin \alpha(H)iP,\,a_*iA )+
\ddt \cos \alpha(H+tA)b_1(H+tA)=
$$
\medskip
$$= - \textstyle \alpha(A) \sin \alpha(H){\rm Im}B(I_0iP,\,\overline L_HiA )+
\ddt \cos \alpha(H+tA)b_1(H+tA)=
$$
\medskip
$$= - \textstyle \alpha(A) \sin \alpha(H){\rm Im}B(A,\,\overline L_H A )+
\ddt \cos \alpha(H+tA)b_1(H+tA)=
$$
\medskip
$$= - \textstyle \alpha(A) \sin \alpha(H)b_1(H)+
\ddt \cos \alpha(H+tA)b_1(H+tA)=
$$
\medskip
$$= -2 \textstyle \alpha(A) \sin \alpha(H)b_1(H)+
\cos \alpha(H)\ddt b_1(H+tA)=0\,.
$$
Equivalently
$$
\ddt b_1(H+tA)=2 \textstyle \frac{\alpha(A) \sin \alpha(H)}
{\cos \alpha(H)}b_1(H)\,.
$$
The solution of this differential equation is 
$\,b_1(H)=ce^{-2 log \cos \alpha(H)}= \frac{c}{\cos^2 \alpha(H)}$, where $\,c\,$ is a real constant.
The initial condition
$\,b_1(0)=0$ forces
$\,c=0\,$ and consequently
 $\,b_1 \equiv 0$.

\nsmallskip
 {\bf \boldmath  $\bullet \ \,a_1\equiv 0$.} In this case we choose the vector fields
$\,\widecheck A$, $\,\widecheck K$ and $\,\widecheck {iA}$. One has 
$$ \textstyle
d\omega_{\mathcal J}(\widecheck A_z,\,\widecheck K_z,\,\widecheck {iA}_z)
=\ddt \omega_{\mathcal J}(\widecheck K_{\exp tAaK^\C},\,\widecheck {iA}_{\exp tAaK^\C})+$$
$$\textstyle
- \ddt \omega_{\mathcal J}(\widecheck A_{\exp tKaK^\C},\,\widecheck {iA}_{\exp tKaK^\C} )+
\ddt \omega_{\mathcal J}(\widecheck A_{\exp i(H+tA)K^\C} ,\,\widecheck K_{\exp i(H+tA)K^\C}) 
\,.$$

\smallskip
\noindent
The first term on the right-hand side of the equal sign vanishes by the $\,G$-invariance of $\,\omega_{\mathcal J}$.
Thus one obtains

\smallskip
$$ -\textstyle \ddt \omega_{\mathcal J}(\dds \exp sA{\exp tKaK^\C},\,(\exp tKa)_*iA )+$$
$$\textstyle 
\ddt \omega_{\mathcal J}(\exp i(H+tA)_*A ,\,\dds \exp s K \exp i(H+tA)K^\C) =
$$
\medskip
$$= -\textstyle \ddt \omega_{\mathcal J}(\dds {\exp tK\exp s(A-t[K,\,A] + O(t^2))aK^\C},\,(\exp tKa)_*iA )+$$
$$ \textstyle 
\ddt \omega_{\mathcal J}(\exp i(H+tA)_*A ,\,-\exp i(H+tA)_*\sin \alpha(H+tA)iP) =
$$
\medskip
$$= -\textstyle \ddt \omega_{\mathcal J}(a_*(A+t\alpha(A)F_aP + O(t^2)),\,a_*iA )+$$
$$\textstyle 
\ddt \sin \alpha(H+tA){\rm Im}B(I_0A ,\,\overline L_{H+tA}iP) =
$$
\medskip
$$=  - \textstyle\omega_{\mathcal J}(a_*\alpha(A)F_aP,\,a_*iA )+ 
\ddt \sin \alpha(H+tA){\rm Re}B(P ,\,\overline L_{H+tA}P) =
$$
\medskip
$$=  \textstyle \alpha(A)\cos \alpha(H) {\rm Im}B(I_0P,\,\overline L_H iA )-
\ddt \sin \alpha(H+tA)a_1(H+tA)  =
$$
\medskip
$$=  -\textstyle \alpha(A)\cos \alpha(H) {\rm Re}B(A,\,\overline L_H A )-
\ddt \sin \alpha(H+tA)a_1(H+tA)  =
$$
\medskip
$$=  \textstyle -\alpha(A)\cos \alpha(H) a_1(H)- 
\ddt \sin \alpha(H+tA)a_1(H+tA) =$$
\medskip
$$  =\textstyle -2\alpha(A)\cos \alpha(H) a_1(H)- 
\sin \alpha(H)\ddt a_1(H+tA) =0\,.$$
Equivalently
$$  \textstyle  
\ddt a_1(H+tA) =-\frac{2\alpha(A)\cos \alpha(H)}{\sin \alpha(H)} a_1(H)\,.$$
\medskip
For $\,H\not=0\,$ the solution of this differential equation is 
$\,a_1(H)=ce^{-2log \sin \alpha(H)}= \frac{c}{\sin^2 \alpha(H)}\,$ and, due to the initial condition
$\,a_1(0)=0$,
one has $\,\lim_{H\to 0}a_1(H)=0$. Hence $\,c=0\,$ and 
 $\,a_1 \equiv 0$.

\nsmallskip
 {\bf \boldmath  $\bullet \ \,a_3(H)= -\cos \alpha(H)$.} For this choose
the vector fields $\,\widecheck P$, $\,\widecheck K\,$ and $\,\widecheck {iA}$.
One has 
$$ \textstyle
d\omega_{\mathcal J}(\widecheck P_z,\,\widecheck K_z,\,\widecheck {iA}_z)
=\ddt \omega_{\mathcal J}(\widecheck K_{\exp tPaK^\C},\,\widecheck {iA}_{\exp tPaK^\C})+$$
$$\textstyle
- \ddt \omega_{\mathcal J}(\widecheck P_{\exp tKaK^\C},\,\widecheck {iA}_{\exp tKaK^\C} )+
\ddt \omega_{\mathcal J}(\widecheck P_{\exp i(H+tA)K^\C} ,\,\widecheck K_{\exp i(H+tA)K^\C})\,,
$$
where the first term on the right-hand side of the equal sign vanishes by the $G$-invariance of 
$\omega_{\mathcal J}$. Thus one obtains
$$\textstyle -\ddt \omega_{\mathcal J}(\dds \exp sP \exp tK aK^\C,\,\dds \exp tK \exp i(H+sA)K^\C)+$$
$$+ \textstyle 
\ddt \omega_{\mathcal J}(\dds \exp s P \exp i(H+tA)K^\C ,\,\dds \exp s K \exp i(H+tA)K^\C) =
$$
\medskip
$$= -\textstyle \ddt \omega_{\mathcal J}(\dds {\exp tK\exp s(P-t[K,P] + O(t^2))aK^\C},\,(\exp tKa)_*iA )+$$
$$+ \textstyle 
\ddt \omega_{\mathcal J}(\exp i(H+tA)_*\cos \alpha(H+tA)P ,\,
-\exp i(H+tA)_*\sin \alpha(H+tA)iP) =
$$
(recall that $[K,P]=2A$)
$$=\textstyle \omega_{\mathcal J}(2a_*A,\,a_*iA )$$
$$\textstyle -
\ddt \cos \alpha(H+tA)\sin \alpha(H+tA)\omega_{\mathcal J}(\exp i(H+tA)_*P ,\,
\exp i(H+tA)_*iP) =
$$
$$= \textstyle -2{\rm Im}B(I_0A, \overline L_H iA )+
\ddt \cos \alpha(H+tA)\sin \alpha(H+tA){\rm Im}B(I_0P ,\,
\overline L_{H+tA}iP) =
$$
$$= \textstyle -2{\rm Re}B(P, \overline L_H A )-
\ddt \cos \alpha(H+tA)\sin \alpha(H+tA){\rm Re}B(A ,\,
\overline L_{H+tA}P) =
$$
$$= -\textstyle 2a_3(H)-
\ddt \cos \alpha(H+tA)\sin \alpha(H+tA)a_2(H+tA) =
$$
$$= -\textstyle 2a_3(H)+
\ddt \frac {\cos \alpha(H+tA)}{a_3(H+tA)}\sin \alpha(H+tA) =0\,.
$$
For the last equality we use that, since $\,a_1=b_1 \equiv 0$,
one has $\,a_2=-\frac{1}{a_3}$ (see claim).
Due to the initial condition $\,a_3(0)= -1\,$ and  the fact that $\,\alpha(A)=2$,
it follows that $\,a_3(H)=- \cos \alpha(H)$. This concludes the proof. 
\qed

%-------------------------------------------------------------------
%-----------------STANDARD----------------------------
%-------------------------------------------------------------------

\bigskip
\section{Appendix B: the canonical K\"ahler form and its potential}
\label{STANDA}

\bigskip
Define  $\,\rho_{can}:\Xi \to \R\,$ by 
$$\,\textstyle \rho_{can}(gaK^\C):= \frac{1}{2}B(H,H)\,,$$
for $\,gaK^\C \in \Xi\,$ with
$\,a= \exp iH$,
and set
 $\,\omega_{can}=-dd_J^c \rho_{can}$, where $\,J=J_{ad}$. As  mentioned
in the introduction,  $\,\Xi\,$ can be thought as a $\,G$-invariant domain in 
the cotangent bundle $\,T^*G/K$. In this realization, from the results in 
\cite{GuSt91} and \cite{LeSz91} (see also \cite{Sz91}),  it follows that
$\,\omega_{can}\,$ coincides with  the 
canonical real symplectic form on~$\,T^*G/K$. 

An analogous computation as in Proposition \ref{POTENTIAL} 
gives the following Lie group theoretic realization of $\,\omega_{can}\,$ and of 
the associated moment map on  $\,\Xi \subset G^\C/K^\C$.
 
%---------------------POTENTIAL2
\bigskip
\begin{prop}
\label{POTENTIAL2}
The  function  $\,\rho_{can}\,$  is a $G$-invariant potential of the
canonical symplectic 
form,
determined by  
$$\,\omega_{can}(\widetilde Z_{aK^\C},\,\widetilde W_{aK^\C}):=
-{\rm Im}B\big (Z,\, E_a^{-1}F_a\overline W \big )\,,$$
for $Z,\,W \in \p^\C$.
Equivalently,
$$\,\omega_{can}(a_*Z,\,a_*W)=
-{\rm Im}B\big (F_a^{-1}Z,\, E_a^{-1}\overline W \big )\,.$$
The moment map $\,\mu_{can}:\Xi \to \g^*\,$ associated with $\,\rho_{can}\,$ is given by 
$$\,\mu_{can}(gaK^\C)(X) =B \big (\Ad_{g^{-1}} X ,\,H \big )\,.$$
\end{prop}

 %--------------------------------------
\bigskip
\begin{remark}\label
{PULLSIMPL}
By means of Lemma \ref{FORMS2}$\,($i$\,)$ and Proposition \ref{POTENTIAL2}, one can check that the form $\,\omega_J\,$ is the
pull-back of  $\,\omega_{can}$ via the $\,G$-equivariant map $\,\psi\,$ defined in Section \ref{DEFORMATION}. $($cf. Rem. \ref{PULLI} and   \cite{DaSz97}, Thm.$\,$4.1$)$.
\end{remark}

%------------------------------------------------------------------------------
%---------------------BIBLIOGRAPHY-----------------------------------
%------------------------------------------------------------------------------

\end{document}